\documentclass[11pt, reqno]{amsart}
\usepackage{comment}
\usepackage[utf8x]{inputenc}
\usepackage[english]{babel}
\usepackage[all,cmtip]{xy}

\usepackage[pdftex]{graphicx}
\usepackage[margin=2.8cm]{geometry}

\usepackage{float}
\usepackage{amsfonts}
\usepackage{amsmath}
\usepackage{amssymb}
\usepackage{amsthm}
\usepackage{dsfont}
\usepackage{tikz}
\usetikzlibrary{shapes.geometric} 
\usetikzlibrary{calc} 
\usetikzlibrary{intersections}

\usepackage{romannum} 
\usepackage{caption}
\usepackage{subcaption}
\usepackage{mathtools}

\usepackage{pgfplots}
\pgfplotsset{compat=1.17} 
\usepackage{MnSymbol}
\usepackage{tikz}
 \usepackage{relsize}
\usepackage[labelfont=bf]{caption}

\usepackage[T1]{fontenc}
\usepackage{paralist}
\usepackage{colonequals}
\usepackage{indentfirst}
\usepackage[colorlinks,link color=blue, cite color=blue,pagebackref,pdftex]{hyperref}

\renewcommand\emptyset{\varnothing}
\renewcommand\phi{\varphi}

\newcommand\R{\mathbb{R}}

\newcommand{\inner}[2]{\langle{#1},{#2}\rangle}

\DeclareMathOperator*{\relint}{relint}

\DeclareMathOperator*{\conv}{conv}
\DeclareMathOperator*{\cone}{cone}
\DeclareMathOperator*{\vertex}{vert}

\DeclareMathOperator*{\bis}{bis}
\DeclareMathOperator*{\dist}{dist}
\DeclareMathOperator{\proj}{proj}
\DeclareMathOperator{\homog}{hom}
\DeclareMathOperator{\interior}{int}

\newtheorem{thm}{Theorem}[section]
\newtheorem*{thm*}{Theorem}
\newtheorem{cor}[thm]{Corollary}
\newtheorem{lem}[thm]{Lemma}
\newtheorem{prop}[thm]{Proposition}

\newtheorem{quest}{Question}

\theoremstyle{definition}
\newtheorem{dfn}[thm]{Definition}
\newtheorem{example}[thm]{Example}

\tikzset{
dot/.style = {circle, fill, minimum size=#1,
              inner sep=0pt, outer sep=0pt},
dot/.default = 6pt 
}

\AtBeginDocument{\pagenumbering{arabic}} 

\title[]{Polyhedral combinatorics of bisectors}

\author{Aryaman Jal}
\address{Department of Mathematics, %
KTH Royal Institute of Technology, 100 44 Stockholm, 
Sweden}
\email{aryaman@kth.se}

\author{Katharina Jochemko}
\address{Department of Mathematics, %
KTH Royal Institute of Technology, 100 44 Stockholm, 
Sweden}
\email{jochemko@kth.se}

\keywords{Polyhedral norms, bisectors, polyhedral combinatorics, Wasserstein distance}
\subjclass[2020]{46B20, 52B12, 52A21, 52C45}

\date{\today}

\begin{document}

\maketitle

\begin{abstract}
For any polyhedral norm, the bisector of two points is a polyhedral complex. We study combinatorial aspects of this complex. We investigate the sensitivity of the presence of labelled maximal cells in the bisector relative to the position of the two points. We thereby extend work of Criado, Joswig and Santos (2022) who showed that for the tropical distance function the presence of maximal cells is encoded by a polyhedral fan, the bisection fan. We initiate the study of bisection cones and bisection fans with respect to arbitrary polyhedral norms. In particular, we show that the bisection fan always exists for polyhedral norms in two dimensions. Furthermore, we determine the bisection fan of the $\ell _1$-norm and the $\ell _\infty$-norm as well as the discrete Wasserstein distance in arbitrary dimensions. Intricate combinatorial structures, such as the resonance arrangement, make their appearance. We apply our results to obtain bounds on the combinatorial complexity of the bisectors.
\end{abstract}

\section{Introduction}
   Bisectors are a fundamental concept in metric geometry. Given a metric space, the bisector of two points (also called sites) is the set of all points with equal distance to both sites. Bisectors are intimately related to Voronoi diagrams, a classical and intensively studied topic in convex and computational geometry with applications in a wide range of different areas such as computer graphics, robotics, data analysis, crystallography and meteorology (see, e.g., ~\cite[Chapters 1 and 7]{Aurenhammer} and~\cite[Section 4]{MartiniSurvey}).

   In the present article we investigate combinatorial aspects of bisectors with respect to convex polyhedral norms in $\mathbb{R}^d$. Any such norm is of the  form $\|x\|_P=\min \{\lambda\geq 0\colon x\in \lambda P\}$ where $P$ is a full-dimensional centrally symmetric convex polytope, the unit ball with respect to the norm. The norm $\|x\|_P$ is also commonly referred to as the gauge function or Minkowski functional of the polytope. (See, e.g., \cite{Schneider,siegel1989lectures}.) 
   
   With respect to the Euclidean norm the bisector of two distinct points is always given by a hyperplane. In contrast, for arbitrary convex norms where the unit ball is given by a convex body, bisectors can be less well-behaved, even in $\mathbb{R}^{2}$; for example, they can have non-empty interior (see, e.g. Figure \ref{fig:2_dim_l1_R2}) or zigzag infinitely many times between two parallel lines \cite{corbalan1996geometry}. In fact, among bisectors, ill behaviour is the standard rather than the exception as evidenced by various characterizations of the Euclidean norm via topological properties of bisectors and Voronoi cells, see e.g. \cite{he2013bisectors, blomer2018voronoi}.

   In general, for arbitrary polyhedral norms, the bisector is no longer always a hyperplane but is given by a polyhedral complex~\cite[Proposition 1]{tropvoronoi}. More precisely, for sites $a,b\in \mathbb{R}^d$ and a polyhedral norm $\|.\|_P$, the bisector of $a$ and $b$,
   \[
   \bis\nolimits^P(a,b) \ = \ \left\{x\in \mathbb{R}^d\colon \|x-a\|_P=\|x-b\|_P\right\}\, ,
   \]
   is a polyhedral complex with cells of the form
   \[
   \bis\nolimits _{F,G}^P(a,b) \ = \ \bis\nolimits^P(a,b)\cap (a+C_F)\cap (b+C_G)
   \]
   where $F,G$ are faces of the polytope $P$ and $C_F$ denotes the face cone of a face $F$. Since $\|.\|_P$ is linear on face cones the cells are indeed polyhedra; see Section~\ref{prelim} for precise definitions.

   For sufficiently generic points $a,b$, the bisector $\bis ^P(a,b)$ is a pure polyhedral complex of dimension $d-1$~\cite[Corollary 2]{tropvoronoi}. The maximal cells in the complex are of the form $\bis ^P_{F,G}(a,b)$ where $F,G$ are facets of $P$. However, not every pair of facets gives rise to a maximal cell: $\bis_{F, G}^P(a,b)$ can be empty or of lower dimension. In~\cite{tropvoronoi}, Criado, Joswig and Santos investigated the sensitivity of the presence (non-emptiness) of the cells $\bis^P _{F,G}(a, b)$ to the relative position of the sites $a,b$ in the case of the tropical norm. They showed that the presence/absence of $\bis^P _{F,G}(a, b)$ considered simultaneously for all facet pairs $(F,G)$ is governed by a polyhedral fan, which they called the \textit{bisection fan}, that can be defined in purely combinatorial terms: the presence/absence of the cells only depends on the bisected ordered partition arising from $a-b$~\cite[Theorem 5]{tropvoronoi}.

   In the present work, we initiate a systematic study of bisection fans for arbitrary polyhedral norms. We study combinatorial and geometric properties of the equivalence relation given by the simultaneous emptiness/non-emptiness of $\bis _{F,G}(a,b)$ depending on the direction of $a-b$. This allows for an identification of bisectors that is coarser than the notion of affinely isomorphic polyhedral complexes but still retains important combinatorial information about them. We investigate structural properties of the equivalence classes valid for arbitrary polyhedral norms. In particular, we introduce and study the notion of bisection cones which form the building blocks of equivalence classes. Special attention is given to the $\ell _1$- and the $\ell_\infty$-norm, as well as to polygonal norms and the discrete Wasserstein distance for which we provide explicit descriptions of the bisection fans. Intricate combinatorial structures, such as the resonance arrangement, make their appearance. \\

   \textit{Main contributions:}
    We introduce the notion of bisection cones for arbitrary polyhedral cones. The bisection cone $\mathcal{B}_{F,G}$ is a polyhedral cone that geometrically encodes when a single maximal cell $\bis^P _{F,G}(a,b)$ is present or absent. We provide different geometric descriptions of the bisection cone (Propositions~\ref{raysbfg} and~\ref{projhom}) and discuss symmetry properties (Proposition~\ref{bfg_swap}). 
    
    In particular, our simple conical hull description (Proposition~\ref{raysbfg}) allows us to show that the common refinement of the face fan $\mathcal{F}_P$ of the unit ball $P$ and the fan $\mathcal{H}_P$ induced by the hyperplane arrangement formed by the facet defining hyperplanes of $P$ induce a partition of the points in general position that is coarser than the equivalence relation on bisectors. In particular, if the bisection fan exists, then it refines both these fans (Proposition~\ref{bis_refine}).

 We then use our results to determine the bisection cones and the bisection fans in the case of polygonal norms in the plane, the $\ell _1$- and the $\ell_\infty$-norm, as well as the discrete Wasserstein distance. 
 \begin{itemize}
     \item In case of polygonal norms in the plane, i.e., norms with polygons as the unit ball, the bisection fan is the complete fan whose rays are generated by all vectors of the form $v_i-v_j$ if $v_i,v_j$ are vertices of the unit ball (Theorem~\ref{prop:bisfan_polygon}).
    \item We determine the bisection cones and bisection fan of the $\ell _1$- and the $\ell_\infty$-norm with the cross-polytope and the cube as unit ball, respectively. (Proposition~\ref{cube:bis_cone}, Theorem~\ref{prop:bisfancube}, Proposition~\ref{bis_tree} and Theorem~\ref{prop:bisfantree2}). In both cases, the bisection fan is the common refinement of $\mathcal{F}_P$ and $\mathcal{H}_P$.
    \item In case of the Wasserstein distance (also known as Earth Mover's distance) the corresponding unit ball is the convex hull of the root lattice of type $A$, $e_i-e_j$, for all $i\neq j$. We consider the norm induced on its affine hull $H=\{x\in \mathbb{R}^d\colon \sum x_i =0\}$. We give a combinatorial description of the bisection fan in terms of set systems of heavy/light positive/negative sums (Theorem~\ref{set_sys}) and an explicit description as regions of linearity of sums of piece-wise linear functions (Theorem~\ref{thm:WassersteinPolyhedral}). 
   \end{itemize}
   We use the description of the bisection fans obtained above to give bounds for the number of maximal cells of the bisector if the sites $a,b$ are in general position (Theorem~\ref{prop:complexity}). We are able to give an exact count in case of polygonal norms and the $\ell _\infty$-norm that is polynomial in the number of vertices and dimension, respectively. Interestingly, for polygonal norms this number only depends on the combinatorial type of the polygon and is independent of any metric properties. For the $\ell_1$-norm and the Wasserstein distance we prove an exponential lower bound.\\

   \textit{Related work:}
   Structural properties of bisectors with respect to convex norms have been studied by various authors, see~\cite[Section 4]{he2013bisectors} and references therein. In particular, there exists a large body of research on geometric properties of bisectors with respect to convex norms in two dimensions, where already many interesting phenomena can be observed (see, e.g., \cite{corbalan1996geometry,ma2000bisectors}), as well as on  $\ell_p$-norms (see, e.g.,~\cite{lp1,lp2}).
   
   In \cite{ma2000bisectors}, Ma initiated the study of questions regarding the characterization, construction -- both theoretical and algorithmic -- and complexity of bisectors with respect to polyhedral norms in two and three dimensions. In particular, a precise criterion was given as to when the bisector of two points corresponding to a convex polygonal distance function is homeomorphic to a line. Criado, Joswig and Santos \cite{tropvoronoi} generalised the work of Ma to obtain a fuller characterization of the geometric, combinatorial and topological properties of the bisector in higher dimensions and several sites, with a particular focus on the tropical distance function. An asymmetric version of the polyhedral distance function corresponding to a simplex, was studied by Amini and Manjunath~\cite{amini2010riemann} as well as Manjunath \cite{MR3037988} in relation to tropical variants of Riemann-Roch theory. Recently, Com\u{a}neci and Joswig considered the asymmetric tropical distance function and investigated the  corresponding Voronoi diagrams~\cite{comuaneci2022asymmetric} as well as its applications to phylogenetics~\cite{comuaneci2022parametric}. 
   
   Motivated by applications to machine learning~\cite{WassersteinGAN,WassersteinLearning}, a lot of recent research activity in algebraic statistics has focused on Voronoi diagrams with respect to the discrete Wasserstein distance~\cite{ccelik2021wasserstein,Becedas2024,Optimaltransport}. The unit ball of the discrete Wasserstein distance that we focus on in the present work is given as the convex hull of all vectors $e_i-e_j$, $i\neq j$, where $e_i$ denotes a standard basis vector in $\mathbb{R}^d$. This polytope is dual to the tropical unit ball and agrees with the symmetric edge polytope of the complete graph. The facet structure of symmetric edge polytopes in general is a topic of recent and vigorous investigation in geometric combinatorics~\cite{higashitani2015smooth,Manyfaces,FacesFacets,Fundamentalpolytopes}.

   As we will show in Section~\ref{sec:Wasserstein}, the bisection fan of the discrete Wasserstein norm is intriguing in that not only does it not arise from a hyperplane arrangement but it also contains as a subfan the fan induced by the resonance arrangement $\mathcal{A}_{n}$  which consists of the hyperplanes $\sum_{i \in I}x_{i} = 0$, for $\emptyset \neq I \subseteq [n]$; the universality and Betti numbers of were recently studied by K\"uhne in~\cite{kuhne2023universality}. The resonance arrangement is the hyperplane arrangement that gives rise to the normal fan of the so-called White Whale polytope, the zonotope defined by the line segments $[0, e_{I}]$ for non-empty $I\subseteq [n]$, the number of vertices of which is only known for $n\leq 9$ \cite{chroman2021computations, brysiewicz2021computing}. The difficulty of this enumeration exemplifies the challenge of a straightforward inequality description of bisection cones in the case of the discrete Wasserstein norm in particular. These connections hint at the subtlety of the geometric-combinatorial ideas at play within the concept of the bisection fan.

   \textit{Outline:}
The article is divided as follows. In Section~\ref{prelim}, we cover the necessary background on polyhedral norms,  define the main objects of our interest -- bisectors -- and detail what is known about them. In Section~\ref{sec:bisectioncones}, we introduce the notion of the bisection cone, an object that keeps track of the non-emptiness of a given maximal cell of a bisector, and obtain various descriptions of it. In the subsequent four sections we construct the bisection fan of centrally symmetric polygons, the $\ell_{1}$-norm, the $\ell_{\infty}$-norm and the discrete Wasserstein norm, respectively. In Section~\ref{sec:complexity}, we use our description of the bisection cone to obtain the number of maximal cells in the bisector as a measure of the complexity of the bisector in the four aforementioned cases. We conclude in Section~\ref{sec:outlook} by discussing evidence supporting the existence of the bisection fan for various classical $3$-dimensional centrally symmetric polytopes and commenting on the problem of showing that every centrally symmetric polytope admits a polyhedral bisection fan. 
    \section{Preliminaries}\label{prelim}
    In the following we collect preliminaries on polyhedral distance functions and their bisectors, following~\cite{tropvoronoi}. We assume basic knowledge of polyhedral geometry, see, e.g.,~\cite{de2010triangulations,Grunbaum,Ziegler}. For further reading on convex distance functions, see also~\cite{Aurenhammer}. Since metrics convex distance functions and convex bodies are in one-to-one correspondence, in what follows we will frequently make reference to bisectors of convex bodies and bisectors of the corresponding norms interchangeably.
    
    In what follows, let $P$ always be a full-dimensional (convex) polytope in $\R^d$ containing the origin in its interior. Then $P$ induces a \textbf{polyhedral distance} function defined by \[\dist \nolimits^P(a,b) = \min \{\alpha \geq 0 : b -a \in \alpha P\}\] for all $a,b\in \R^d$. This distance function satisfies the triangle inequality, is invariant under translation and is homogeneous with respect to scaling by a positive real number. When $P$ is \textbf{centrally symmetric}, that is, $P = -P$, this distance function is indeed a metric in the usual sense and $\dist^P(0,.)$ is a norm. The \textbf{bisector} $\bis^{P}(S)$ of a finite set of points $S$ is defined as 
    \[\bis\nolimits ^P(S) = \left\{x \in {\R}^{d}: \dist \nolimits^P(a, x) = \dist \nolimits ^P(b, x) \quad \text{for $a, b \in S$} \right\}.
    \]
    Throughout we will restrict ourselves to centrally symmetric polytopes and bisectors of two sites, though some of the results can be extended beyond that. When the polytope under consideration is clear from context we will often suppress $P$ in the notation of the distance function and bisector.
    
    Given a face $F$ of $P$, the \textbf{face cone} of $F$ is 
    \[
    C_{F} = \text{cone}(F) = \left \{\sum \limits_{i=1}^{k}\lambda_{i}x_{i}: x_{i} \in F, \lambda_{i} \geq 0 , k \in \mathbb{N}\right \}.
    \]
    The collection $\mathcal{F}_{P} = \{\text{cone}(F): F \subsetneq P, \, \, F \, \, \text{face} \}$ forms the \textbf{face fan} of $P$. By definition, the restriction of $\dist^P(0, \cdot)$ to each face cone is linear. 

        Given a pair of points $(a,b)$ and faces $F,G$ of $P$ let 
        \[
        \bis \nolimits_{F,G}^P(a,b) \coloneqq\bis\nolimits^P(\{a, b\}) \cap (C_F+a) \cap (C_G+b) \, .
        \] 
        By the piecewise linearity of $\dist^P(0, \cdot)$, $\bis \nolimits_{F,G}^P(a,b)$ is always a polyhedron. These polyhedra define a polyhedral subdivision of the bisector.
   \begin{prop}[{\cite[Proposition 1]{tropvoronoi}}]\label{intro_cells_of_bis}
    Let $P\subset \mathbb{R}^d$ be a full-dimensional centrally symmetric polytope. Then $\bis^P(a,b)$ is a polyhedral complex each of whose cells are the polyhedra $\bis \nolimits_{F,G}^P(a,b)$
    for all choices of faces $F, G$ of $P$.
    \end{prop} 

See Figure~\ref{fig:1_dim_l1_R2} for an example of the polyhedrality of the bisector under the $\ell_{1}$-norm in dimension~$2$. Figure~\ref{fig:2_dim_l1_R2} shows an instance of full-dimensional cells, the presence of which is explained ahead.

There will certainly be pairs of facets $(F, G)$ for which the cells $\bis \nolimits_{F,G}^P(a,b)$ will be empty. We will investigate the emptiness/non-emptiness of $\bis \nolimits_{F, G}^{P}(a, b)$ condition given a pair of facets $(F, G)$ of the polytope. Relating bisectors to one another on the basis of this condition was the motivation for the following notion of equivalence of bisectors; this was first considered in~\cite{tropvoronoi} in the case of the tropical norm. 

\begin{dfn}\label{def:equivalencebisectors}
    Two bisectors $\bis^P(a,b)$ and $\bis^P(a',b')$ are said to be \textbf{equivalent} if
\begin{equation}\label{def:equivalence}
    \bis \nolimits _{F,G}^P(a,b) \neq \emptyset \Leftrightarrow \bis \nolimits_{F,G}^P(a',b') \neq \emptyset \tag{E}
\end{equation}

 for all choices of facets $F,G$. 
\end{dfn}
This defines an equivalence relation on pairs of points $(a,b)$ in $\R^d \times \R^d$. Henceforth, we will use these two notions of equivalence interchangeably. We will study the equivalence classes as geometric sets and aim for a  polyhedral understanding of this equivalence relation. Of fundamental importance in this context is the notion of general position. A pair of points $(a,b)$ is said to be in \textbf{weak general position} if $a-b$ is not contained in any hyperplane parallel to a facet of $P$. The following result explains the geometric relevance of this notion. 
 \begin{prop}[{\cite[Proposition 2]{tropvoronoi}}]\label{prop:weakgeneral}
 For any two points $a,b\in \R^d$, $bis^P(a,b)$ contains full-dimensional cells if and only if $a-b$ is contained in a hyperplane parallel to a facet of $P$.
 \end{prop}
 In fact, it can be seen that $\bis _{F,F}^P(a,b)$ is full-dimensional if and only if $a-b$ is contained in a hyperplane parallel to the facet $F$; otherwise $\bis _{F,F}^P(a,b)$ is empty. See Figure~\ref{fig:2_dim_l1_R2} for an example.
 
 A pair of points $(a,b)$ is said to be in \textbf{general position} if there exist open neighborhoods $U$ of $a$ and $W$ of $b$ such that $\bis^P(a,b)$ and $\bis^P(a',b')$ are equivalent for all $a'\in U$ and $b'\in W$. In other words, $(a,b)$ is in general position if its equivalence class of $\bis^{P}(a, b)$ under the relation specified given in~\eqref{def:equivalence} remains unchanged under small perturbations, that is, $(a,b)$ lies in the interior of its equivalence class. By Proposition~\ref{prop:weakgeneral} any pair of points in general position is also in weak general position. 
 
 The general position of points has strong geometric implications, as the next result on the presence of a single cell $\bis _{F,G}^P(a,b)$ demonstrates.
  \begin{thm}[{\cite[Theorem 1]{tropvoronoi}}]\label{thm:generalposition}
  Given $(a,b)\in \mathbb{R}^2$ and facets $F,G$ of $P$. Let $\lambda _F,\lambda _G$ be the restrictions of $\dist(0,\cdot)$ to $F$ and $G$, respectively. Then the following are equivalent.
\begin{itemize}
    \item[(i)] There are neighborhoods $U$ and $W$ of $a$ and $b$, respectively, such that for any choice of $a'$ in $U$ and $b'$ in $W$ the cell $\bis _{F,G}^P(a',b')$ is non-empty;
    \item[(ii)] $(C_{F}+a)\cap (C_{G}+b)$ is full-dimensional and is intersected in the interior by the (non-degenerate) hyperplane $H=\{x\in \mathbb{R}^d\colon \lambda _F(x-a)=\lambda _G (x-b)\}$.
\end{itemize}
 \end{thm}
     Hence for points $(a, b)$ in general position, each maximal cell $\bis^{P}_{F, G}(a, b)$ is either $(d-1)$-dimensional or empty. Consequently, the following holds.
 \begin{cor}[{\cite[Corollary 2]{tropvoronoi}}]
 The bisector $\bis^P(a,b)$ of two points in general position is a pure polyhedral complex of dimension $d-1$.
 \end{cor}

 We observe that, up to translation, $\bis(a,b)$ equals $\bis(0,b-a)$ for all $a,b$. Further, $\bis^P_{F,G}(a,b)$ is non-empty if and only if $\bis^P(ta,tb)$ is non-empty for all $t>0$. It follows that the equivalence class of $\bis^P(a,b)$ is completely determined by the direction of $b-a\in \mathbb{R}^d$. In particular, it induces an equivalence relation on $\mathbb{R}^{d}$. Similarly as for pairs of points, we also call a point $c\in \R^d$ in general position if there exists a neighborhood $U$ of $c$ such that  $\bis^P(0,c)$ is equivalent to $\bis^P(0,c')$ for all $c'\in U$. It follows that a pair of points $(a,b)$ is in general position if and only if $b-a$ is in general position.

 The equivalence relation~\eqref{def:equivalence} captures rudimentary combinatorial properties of bisectors, namely the emptiness or non-emptiness of its cells $\bis_{F,G}^P(a,b)$. Since for any $a\in \mathbb{R}^d$, $\bis^P_{F,G}(0,a)$ is non-empty if and only if $\bis^P(0,ta)$ is non-empty for all $t>0$, the equivalence classes are closed under scaling. From a discrete-geometric perspective it is natural to ask if the equivalence classes have an interesting polyhedral structure, in particular, if the equivalence classes are given by polyhedral cones that have a combinatorial meaning. In~\cite[Theorem 5]{tropvoronoi} it was shown that this is indeed the case for the tropical norm. In this case the equivalence class of $\bis (0,a)$ is determined by the so-called bisected ordered partition induced by the coordinates of $a$. The bisected ordered partitions induce a fan structure on $\mathbb{R}^d$, called the bisection fan in~\cite{tropvoronoi}. We generalize this concept to arbitrary polyhedral norms.

\begin{dfn}\label{bisfandef}
 Let $P$ be a centrally symmetric polytope and let $\Delta$ be a polyhedral fan. Then $\Delta$ is the \textbf{bisection fan} of $P$ if and only if for all $a,b\in \R^d$ in general position the following two conditions are equivalent.
 \begin{itemize}
     \item[(i)] The bisectors $\bis^P(0,a)$ and $\bis^P(0,b)$ are equivalent.
     \item[(ii)] Both points $a$ and $b$ lie in the interior of the same maximal cone of $\Delta$.
 \end{itemize}
 \end{dfn}
 
The question of whether or not the bisection fan exists for every polyhedral distance function is subtle and does not seem to be easily resolved, a priori. We approach this question by investigating geometric and combinatorial properties of a natural family of cones, the bisection cones that we introduce in the next section. They form the constituent pieces of the equivalence classes induced by the relation in Equation \eqref{def:equivalence}. In subsequent sections, we then show that the bisection fan exists for polygonal norms in $2$ dimensions, the $\ell_1$- and the $\ell_\infty$-norm as well as the discrete Wasserstein norm in arbitrary dimensions.

\begin{figure}[H]
    \centering
    \begin{tikzpicture}
\definecolor{vertexcolor_unnamed__1}{rgb}{ 1 0 0 }
 \definecolor{facetcolor_unnamed__1}{rgb}{ 0.4666666667 0.9254901961 0.6196078431 }
    \node [
    regular polygon, 
    regular polygon sides=4, 
    minimum size=4 cm, shape border rotate = 45, scale = 1,
    draw=black, ultra thick
] 
at (0,0) (A){};
\node[dot=4pt, fill=black, label=left:{$0$}] at (A.center) {};
\node [
    regular polygon, 
    regular polygon sides=4, 
    minimum size=4 cm, shape border rotate = 45, scale = 1,
    draw=black, thick
] 
at ($(15:1.25)+(0, 1.4)$) (B){};
\node[dot=4pt, fill=black, label={[xshift=0.25cm, yshift=0.25cm]{$a$}}] at (B.center) {};
\node[dot=0pt, fill=black, label={[xshift=0.85cm, yshift=-0.75cm]{$1$}}] at (A.corner 1) {};
\node[dot=0pt, fill=black, label={[xshift=0.6cm, yshift=0.8cm]{$2$}}] at (A.corner 2) {};
\node[dot=0pt, fill=black, label={[xshift=-1.5cm, yshift=0.8cm]{$3$}}] at (A.corner 3) {};
\node[dot=0pt, fill=black, label={[xshift=1.5cm, yshift=0.8cm]{$4$}}] at (A.corner 3) {};
\draw [shorten >= -3.15cm, shorten <=0cm] 
(A.center) [dotted, thick] edge (A.corner 4);
\draw [shorten >= -2cm, shorten <=0cm] 
(B.center) [dotted, thick] edge (B.corner 4);

\draw [shorten >= -1cm, shorten <=-2cm] (B.center) [dotted, thick] edge (B.corner 1);
\path (A) -- (B) coordinate[midway] (M);

\path[name path=a_to_a_corner1] (A.center)--(A.corner 1); 
\path[name path=m_135] (M) -- ($(M)+(135:1)$);
\fill [red, name intersections={of=a_to_a_corner1 and m_135}] (intersection-1) circle (2pt);
\coordinate (S) at (intersection-1);
\node[dot=4pt, fill=red] at (S){};
\draw[blue, thick, line width = 0.5mm] (M) -- (S);
\draw[blue, thick, line width = 0.5mm] (S) -- ($(S)+(-3, 0)$) node[below, label={[xshift=0cm, yshift=0cm]{$\bis_{2,3}(0, a)$}}] {};

\path[name path=b_to_b_corner3] (B.center)--(B.corner 3); 
\path[name path=m_315] (M) -- ($(M)+(-45:2)$);
\fill [blue, name intersections={of=b_to_b_corner3 and m_315}] (intersection-1) circle (2pt);
\coordinate (T) at (intersection-1);
\draw[blue, thick, line width = 0.5mm] (M) -- (T) node[below] {};
\draw[blue, thick, line width = 0.5mm] (T) -- ($(T)+(3, 0)$) node[below, label={[xshift=-0.2cm, yshift=0cm]{$\bis_{1,4}(0, a)$}}] {};

\path[name path=a_to_a_corner4] (A.center)--(A.corner 4); 
\path[name path=b_to_b_corner_3] (B.center) -- ($(B.corner 3)+(0, -1)$);
\fill [blue, name intersections={of=a_to_a_corner4 and b_to_b_corner_3}] (intersection-1) circle (0pt);
\coordinate (Z)  at (intersection-1);
\path let \p2= ($(A.center)+(4, 0)$), \p3= ($(B.center)+(4, 0)$) in [fill=blue,opacity=0.05] (B.center) -- (Z) -- (\x3, \y2) -- ($(B.center)+(4, 0)$) -- cycle;
\node [
    regular polygon, 
    regular polygon sides=4,
    minimum size= 2.9 cm, shape border rotate = 45, scale = 1,
    draw=red, very thick, label={[xshift=0cm, yshift=-1.45cm]$x$}
] 
    at (S) (C){};

\draw (A.corner 3) [dotted, thick] edge (S);
\draw (A.corner 1) [dotted, thick] edge ($(A.corner 1)+(0, 2)$);

\end{tikzpicture}
\caption{The unit ball $P$ of the $\ell_{1}$-norm with facets labelled $1,2,3,4$ with thickened edges. The bisector $\bis^{P}(0, a)$ is marked in blue. The point $x$ lies in $\bis^{P}(0, a)$ since the red, scaled unit ball centered at $x$ contains $0$ and $a$ simultaneously in the boundary. The shaded region is $C_{1} \cap (C_{4}+a)$, the intersection of the cones over the facet $1$ and $4$ with apices $0$ and $a$, respectively. The maximal cell $\bis_{1, 4}(0, a)$ is contained in this region.}
    \label{fig:1_dim_l1_R2}
\end{figure}
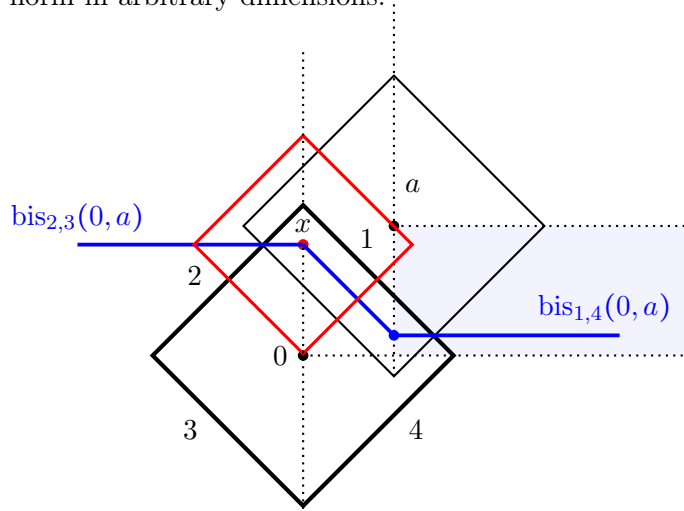

\begin{figure}
    \centering
    \vspace{-0.2 cm}
    \scalebox{0.75}{
    \begin{tikzpicture}
\definecolor{vertexcolor_unnamed__1}{rgb}{ 1 0 0 }
 \definecolor{facetcolor_unnamed__1}{rgb}{ 0.4666666667 0.9254901961 0.6196078431 }
    \node [
    regular polygon, 
    regular polygon sides=4, 
    minimum size=4 cm, shape border rotate = 45, scale = 1,
    draw=black, thick
] 
at (0,0) (A){};
\node[dot=4pt, fill=black, label={[font=\Large]left:{$a$}}] at (A.center) {};
\node [
    regular polygon, 
    regular polygon sides=4, 
    minimum size=4 cm, shape border rotate = 45, scale = 1,
    draw=black, thick
] 
at (45:4) (B){};
\draw [shorten >= -3cm, shorten <=-3cm] (A.center) [dotted, thick] edge (B.center);
\node[dot=4pt, fill=black, label=left:{\Large $b$}] at (B.center) {};
\path (A) -- (B) coordinate[midway] (M);

\path[name path=ac1] (A.corner 1)--($(A.corner 1)+(0, 3)$); 
\path[name path=bd1] (M)--($(M)+(-3, 3)$); 
\fill [blue, name intersections={of=ac1 and bd1}]
(intersection-1) circle (2pt);
\coordinate (q)  at (intersection-1);

\path[name path=ac2] (B.corner 3)--($(B.corner 3)+(0, -3)$); 
\path[name path=bd2] (M)--($(M)+(3, -3)$); 
\fill [blue, name intersections={of=ac2 and bd2}]
(intersection-1) circle (2pt);
\coordinate (r)  at (intersection-1);

\coordinate (s) at ($(q)+(-4,0)$);
\coordinate (t) at ($(r)+(4,0)$);

\draw[blue, thick, line width = 0.5mm] (q) -- (M);
\draw[blue, thick, line width = 0.5mm] (r) -- (M);
\draw[blue, thick, line width = 0.5mm] (q) -- ($(q)+(0, 4)$) node[below, label={[xshift=1.3cm, yshift=-1.8cm] \Large $\bis^{P}(a, b)$}] {};
\draw[blue, thick, line width = 0.5mm] (q) -- (s);
\draw[blue, thick, line width = 0.5mm] (r) -- ($(r)+(0, -4)$);
\draw[blue, thick, line width = 0.5mm] (r) -- (t);
\path[fill=blue,opacity=0.1] (q) -- ($(q)+(0, 4)$) -- (s) -- cycle;
\path[fill=blue,opacity=0.1] (r) -- ($(r)+(0, -4)$) -- (t) -- cycle;
\end{tikzpicture}
}
    \caption{A $2$-dimensional bisector of two points under the $\ell_{1}$-norm in dimension $2$. Since $a-b$ is parallel to a facet of the unit ball, $\bis(a, b)$ contains full-dimensional cells.}
    \label{fig:2_dim_l1_R2}
\end{figure}
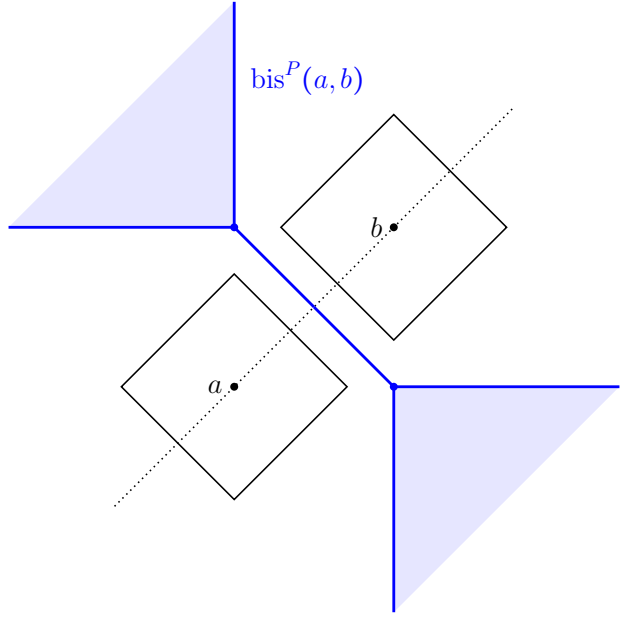
     \section{Bisection cones and bisection fans}\label{sec:bisectioncones}
    In this section we collect structural results concerning the bisection fan and related objects.
    
    Given a polytope $P$ and facets $F$ and $G$ of $P$, the \textbf{bisection cone} (with respect to the pair $(F,G)$) is defined as
    \[
    \mathcal{B}_{F, G} = \left\{x \in \mathbb{R}^{d}: \bis \nolimits_{F, G}(0, x) \neq \emptyset\right\}. 
    \]
    The bisection cones form the building blocks for the equivalence classes of bisectors in the following sense. By Definition~\ref{def:equivalencebisectors}, two points $a,b\in \mathbb{R}^d$ are equivalent if and only if $\bis_{F,G}(0, a)$ and $\bis_{F,G}(0, b)$ are simultaneously empty/non-empty for all facets $F,G$ of $P$. The latter condition can be restated as $a \in \mathcal{B}_{F, G}$ if and only if $b \in \mathcal{B}_{F, G}$ for every facet pair $(F, G)$. Thus, the equivalence classes are given precisely by sets of the form  

    \begin{equation}\label{e_s}
       E_{S} = \bigcap _{(F,G)\in S}\mathcal{B}_{F,G} \cap \bigcap_{(F,G)\in S^c}\mathcal{B}_{F,G}^c
    \end{equation}
    
    where $S$ is a subset of the set of pairs of facets and $\mathcal{B}_{F,G}^c=\mathbb{R}^d\setminus \mathcal{B}_{F,G}$ denotes the complement of the bisection cone.  
    
    The bisection cone $\mathcal{B}_{F,G}$ is always a non-empty polyhedral cone. One way to see this would be to consider the \textbf{homogenization} which is the polyhedral cone that is defined by
    \[\homog(P) = \cone(P \times \{1\}) = \{(x, \lambda) \in \mathbb{R}^d\times \mathbb{R}_{\geq 0}: x \in \lambda P\} \subset \mathbb{R}^{d+1}\]
    for any polytope $P\subset \mathbb{R}^d$. If $F$ is a facet, then the homogenization of $F$ captures the distance $\dist^P(0,x)$ for any point $x$ in the facet cone $C_F$:
    \[
    \homog(F) \ = \ \{ (x,\dist\nolimits ^P(0,x)): x \in C_F\} \, .
\]

With this, we have the following description of the bisection cone.

    \begin{prop}\label{projhom}
    Let $P\subset \mathbb{R}^d$ be a polytope having the origin in its interior and let $F,G$ be facets of $P$. Then
    \[
    \mathcal{B}_{F,G} \ = \ \proj \left((\homog(F)+(-\homog(G)))\cap \{x \in \mathbb{R}^{d+1}: x_{d+1} = 0\}\right)
    \]
    where $\proj\colon \mathbb{R}^{d+1}\rightarrow \mathbb{R}^d$ denotes the projection onto the first $d$ coordinates. 
        Moreover, for any facet $F=\{x\in P\colon z^Tx=b\}$  \begin{eqnarray*}
    \mathcal{B}_{F,-F} &=& C_F\\
    \mathcal{B}_{F,F} &=& \{x\in \mathbb{R}^d\colon z^Tx=0\} \, .
    \end{eqnarray*}
    In particular, $\mathcal{B}_{F,G}$ is a non-empty polyhedral cone for all facets $F,G$ of $P$.
    \end{prop}
    
    \begin{proof}
    A point $x\in \mathbb{R}^d$ is contained in $\mathcal{B}_{F,G}$ if and only if there exists a $y\in C_F\cap (C_G + x)$ with $\dist(0,y)=\dist(x,y)=\dist(0,y-x)$. In this case, $(y,\dist(0,y))$ is contained in the intersection $\homog(F) \cap (\homog(G)+(x,0))$. In particular, the intersection is non-empty and 
    \[(x,0)\in \homog (F)+(-\homog (G))\] which is equivalent to $x$ being contained in $\proj \left((\homog(F)+(-\homog(G)))\cap \{x \in \mathbb{R}^{d+1}: x_{d+1} = 0\}\right)$. This proves the first part of the proposition.
    
    For the second part, let $F$ and $-F$ be opposite facets of $P$. For any $x\in \mathbb{R}^d$ such that $(x,0)\in (\homog(F)+(-\homog(-F))$ we have that $x=a+(-(-b))$ for some $a,b\in C_F$. In particular, $x$ is contained in the cone $C_F$. This shows that $\mathcal{B}_{F,-F}$ is contained in $C_F$. For the other inclusion we observe that for any $x\in C_F$ we can decompose $(x,0)$ as
    \[
    (x,0) \ = \ \left(\tfrac{x}{2}, \dist(0,\tfrac{x}{2})\right)+\left(\tfrac{x}{2}, -\dist(0,\tfrac{x}{2})\right) \in \hom (F)+(-\homog(-F)) \, .
    \]
    
    Further, we observe that for any facet $F$,
    
    \[
    \homog(F) = \left\{ \left(\sum_{i=1}^{m}\lambda_{i}u_{i}, \sum_{i=1}^{m}\lambda_{i} \right): u_{i} \in F, \lambda_i \geq 0 \right\}
    \]
    Thus, 
    \begin{eqnarray*}
           \mathcal{B}_{F, F} &=& \proj \left((\hom(F)+(-\hom(F)))\cap \left\{x \in \mathbb{R}^{d+1}: x_{d+1}=0 \right\}\right)\\ &=& \left\{\sum_{i=1}^{m}(\lambda_{i} - \mu_{i})u_{i}: \sum_{i=1}^{m}(\lambda_{i}-\mu_{i}) = 0, u_{i} \in F, \lambda _i, \mu_i \geq 0 \right\}\\
           &=& \left \{\sum_{i=1}^{m}c_{i}u_{i}: \sum_{i=1}^{m}c_{i} = 0, u_{i} \in F \right\} 
    \end{eqnarray*}
    which is equal to $\{x \in \mathbb{R}^{d}: z^{T}x =0\}$, if $F=\{x\in P\colon z^Tx=b\}$.\qedhere \end{proof}

Let $\vertex (F)$ denote the set of vertices of a facet $F$. We have the following conical hull description of bisection cones.

\begin{prop}\label{raysbfg}
Let $P$ be a polytope and let $F, G$ be facets of $P.$ Then \[
\mathcal{B}_{F, G} = \cone(\{v - u: v \in \vertex(F), u \in \vertex(G)\}).
\]
\end{prop}   

\begin{proof}
Let $\vertex(F)=\{p_1,\ldots, p_m\}$ and $\vertex(G)=\{q_1,\ldots, q_n\}$. By Proposition~\ref{projhom} it suffices to prove 
\[
(\homog(F)+(-\homog(G)))\cap \{x \in \mathbb{R}^{d+1}: x_{d+1} = 0\} = \cone\{ (p_{i} - q_{j}, 0): i=1, \ldots m, j =1, \ldots n\} \, .
\]

$\supseteq$: For all $1\leq i\leq m$ and $1\leq j\leq n$ we have $(p_i,1)\in \hom (F)$ and $(q_j,1)\in \hom (G)$. Thus, $(p_i-q_j,0)=(p_i,1)-(q_j,1)$ is contained in $(\homog(F)+(-\homog(G)))\cap \{x \in \mathbb{R}^{d+1}: x_{d+1} = 0\}$. Since the latter is a cone, the conical hull of all $(p_i-q_j,0)$ is also contained in it.

$\subseteq$: For the other inclusion, let $(w, 0) \in (\homog(F)+(-\homog(G)))\cap \{x \in \mathbb{R}^{d+1}: x_{d+1} = 0\}.$ If $w=0$ then $(w,0)$ is clearly also contained in the right hand side. In the other case, there exist coefficients $\lambda_{i}, \mu_{j} \geq 0$ for $ i =1, \ldots m$ and $j=1, \ldots ,n$, not all zero, such that 
\[
(w, 0) = \sum_{i=1}^{m}\lambda_{i}(p_{i}, 1) + \sum_{j=1}^{n}\mu_{j}(-q_{j}, -1)\, .
\]
Then the sum of the $\lambda_i$s and $\mu _j$s must necessarily be the same. Let $\sum_{i=1}^{m}\lambda_{i} = \sum_{j=1}^{n}\mu_{j}:=M >0$. Then
\[
\dfrac{w}{M} \ = \ \sum_{i=1}^{m}\dfrac{\lambda_{i}}{M}p_{i} -  \sum_{j=1}^{n}\dfrac{\mu_{j}}{M}q_{j} \ = \ \sum_{i, j}\dfrac{\lambda_{i}}{M}\dfrac{\mu_{j}}{M}(p_{i} - q_{j}) \, .
\]

This implies $(\frac{w}{M}, 0) \in \cone \{(p_{i} - q_{j}, 0): i=1, \ldots m, j =1, \ldots n\}$ and hence $(w, 0)$ is also contained in the right hand side.\qedhere
\end{proof}

As an immediate consequence of Proposition~\ref{raysbfg} we obtain the following symmetry properties of the bisection cone.
\begin{cor}\label{bfg_swap}
    For any facets $F,G$ of $P$ we have 
    \[
    \mathcal{B}_{G, F} \ = \ -\mathcal{B}_{F, G} \ = \ \mathcal{B}_{-F, -G} \, .
    \]
\end{cor}

We record the possible dimensions of bisection cones in the following proposition.
    
       \begin{prop}\label{dim_bfg}
    Let $F,G$ be facets of a polytope $P$. Then
    \[
    \dim \mathcal{B}_{F,G} \ = \ \begin{cases} d & \text{if} \, \, F\neq G, \\d -1 & \text{if} \, \, F= G.
        \end{cases}
        \]

    \end{prop}
    
    \begin{proof}
      If $F=G$, then $\mathcal{B}_{F, G}$ is a hyperplane by Proposition \ref{projhom}. Otherwise, if $F \neq G$, we can find vertices $v_{1}, \ldots , v_{d}$ of $F$ and a vertex $v_{d+1}$ of $G$ such that $\{v_{i}\}_{i=1}^{d+1}$ is affinely independent. In other words, the set $\{v_{i} - v_{d+1}: i=1, \ldots , d\}$ is linearly independent and contained in $\mathcal{B}_{F, G}$ by Proposition \ref{raysbfg}. This implies that the dimension of $\mathcal{B}_{F, G}$ is $d$. 
      \end{proof}

The following decomposition of the half-space $H^{\geq}$ into bisection cones hints towards geometric restrictions inherent in the equivalences classes of~\eqref{def:equivalence}. 
      
      \begin{prop}\label{h_decomp}
Let $F = \{x \in P: z^{T}x = b\}$ be a facet of $P$, with $z^{T}x \leq b$ a valid inequality for~$P.$ Then \[
H^{\geq} \ := \ \{ x \in \mathbb{R}^{d}: z^{T}x \geq 0\} \ = \ \bigcup_{G \neq F}\mathcal{B}_{F, G} \, . 
\]
\end{prop}

\begin{proof}
To show that $\mathcal{B}_{F, G} \subseteq H^{\geq}$ for every facet $G\neq F$, by Proposition ~\ref{raysbfg} it suffices to show that $v - u \in H^{\geq}$ for every $v \in \vertex(F), u \in \vertex(G).$ This holds since for every $v \in \vertex (F) \subset F$ and any $u\in P$ we have $z^{T}u \leq b=z^{T}v$.

To see the opposite inclusion, we consider a point $y \in \relint(-F)$ and for each facet $G \neq F$ the cone over $-G$ with apex $y$,  $K_{-G} \colonequals \{y + \lambda(x-y): x \in -G, \lambda \geq 0\}$. By coning from $y$ over every facet $-G\neq -F$ we obtain the subdivision
\begin{equation}\label{eq:bis_cone_subdiv}
H^{\geq}+y = \bigcup_{G \neq F}K_{-G} \, .
\end{equation}

See Figure \ref{fig:hexagon_subdiv} for an illustration of this subdivision which is induced by the pulling refinement of $y$ in the trivial subdivision of $P$, see~\cite[Section 4.3.4]{de2010triangulations}.

\begin{figure}[h]
    \centering
    \begin{tikzpicture}
    \node [
    regular polygon, 
    regular polygon sides=6, 
    minimum size=4 cm, shape border rotate = 30, scale = 1,
    draw=black, thick
] 
at (0,0) (A){};
\node[dot=4pt, fill=black, label=right:{}] at (A.center) {};
\foreach \i in {1,2, 3}
    \node[circle, label=left:$v_{\i}$] at (A.corner \i) {};
\foreach \i in {4}
    \node[circle, label=below:$v_{\i}$] at (A.corner \i) {};
\foreach \i in {5}
    \node[circle, label=below:$v_{\i}$] at (A.corner \i) {};
\foreach \i in {6}
    \node[circle, label=right:$v_{\i}$] at (A.corner \i) {};
\node[dot=4pt, fill=black, label=below:{$y$}] at ($0.67*(A.corner 4) + 0.33*(A.corner 3)$) (Y) {};
\foreach \i in {1,2,5,6}
    \draw [->] (Y) -- ($1.5*(A.corner \i)-0.5*(Y)$);
\foreach \i in {3}
    \draw [->] (Y) -- ($2.5*(A.corner \i)-1.5*(Y)$);
\foreach \i in {4}
    \draw [->] (Y) -- ($5*(A.corner \i)-4*(Y)$);
\draw [shorten >= -4cm, shorten <=-4cm] (A.corner 2) [dotted, thick] edge (A.corner 5) node[right, label={[xshift= 7.3 cm, yshift = -4.5 cm]{}}, label={[xshift= 6.3 cm, yshift = -3.5 cm]{$H^{\geq}$}}]  {};
\node[dot=0pt, fill=black, label=above:{$F$},label={[xshift=-2cm, yshift=1.5cm]{$K_{-G}$}}, label={[xshift=-2.3cm, yshift=0cm]{$-G$}}] at ($0.5*(A.corner 1) + 0.5*(A.corner 6)$) (Z) {};
\node[dot=0pt, fill=black, label={[xshift=-0.3cm, yshift=-0.45cm]{$-F$}}] at ($0.5*(A.corner 3) + 0.5*(A.corner 4)$) (X) {};
\draw[->] ($2*(A.corner 5)$)  -- ($2*(A.corner 5) + (60:1cm)$);
\path[fill=blue,opacity=0.1] (Y) -- ($1.5*(A.corner 2)-0.5*(Y)$) -- ($1.5*(A.corner 1)-0.5*(Y)$) -- cycle;

\end{tikzpicture}
    \caption{A subdivision of $H^{\geq}$ into the cones $y+K_{-G}$. Translating $y$ to the origin, this subdivision is the decomposition in Equation \eqref{eq:bis_cone_subdiv}.}
    \label{fig:hexagon_subdiv}
\end{figure}
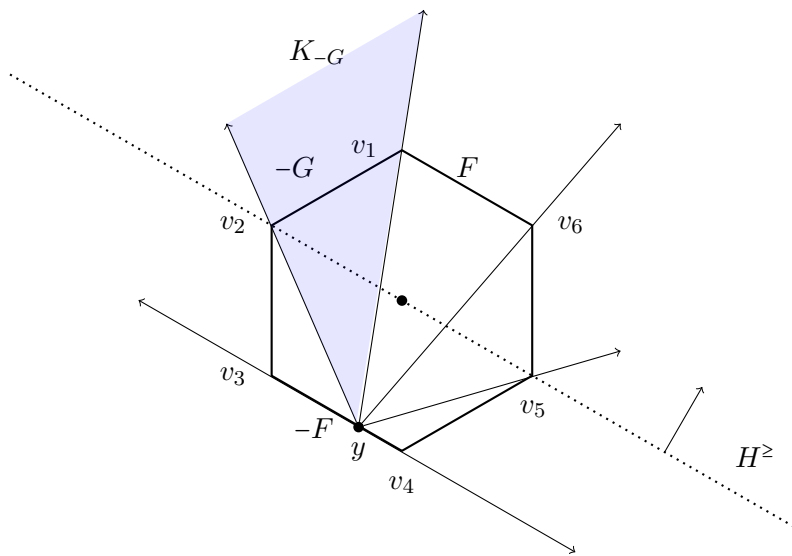

In particular, for any $x \in H^{\geq}$ there exists a facet $G$ such that $x$ is contained in $K_{-G} -y$. This cone has apex $0$ and is generated by $u_{j}-y, j=1,\ldots,r$ where each $u_{j}$ is a vertex of $-G$. That is $x=\sum _{j=1}^r\lambda _j (u_j-y)$ for some $\lambda _1,\ldots, \lambda _r\geq 0$. Further, since $y$ lies in $-F$, we can express $y$ as $y = \sum_{i=1}^{n}\mu_{i}v_{i}$ where $0\leq \mu_{i} \leq 1$ for $i=1, \ldots, n$ with $\sum_{i=1}^{n}\mu_{i} = 1$ and $v_{i}$ are the vertices of $-F$. Thus
\[
x \ = \ \sum _{j=1}^r \lambda _j ( u_{j}-y) \ = \ \sum_{j=1}^r \lambda _j \left(\sum_{i=1}^n \mu _i (u_{j}-v_i)\right)\ = \ \sum_{j=1}^r \sum_{i=1}^n \lambda _j \mu _i (u_{j}-v_i)
\]
which is contained in the bisection cone $\mathcal{B}_{F,G}$ by Proposition~\ref{raysbfg}.
\end{proof}

Let $\mathcal{H}_{P}$ be the hyperplane arrangement consisting of all hyperplanes passing through the origin that are parallel to a facet of $P$. By abuse of notation we use the same notation also for the fan induced by the hyperplane arrangement. As a consequence of Propositions~\ref{projhom} and~\ref{h_decomp} we have the following.

\begin{prop}\label{lem:bis_refine}
    Let $a,b$ be points in general position. If $a$ and $b$ are separated by a hyperplane in $\mathcal{H}_P$ or lie in different maximal cells of $\mathcal{F}_P$ then $\bis (0,a)$ and $\bis(0,b)$ are not equivalent.
\end{prop}
\begin{proof}
Let $H\in \mathcal{H}_P$ be a hyperplane parallel to a facet $F$ of $P.$ If $a$ and $b$ are separated by $H$ then, by Proposition~\ref{h_decomp}, $a\in \mathcal{B}_{F,G}$ for some $G\neq F$ but $b\not \in \mathcal{B}_{F,G}$, or the other way round, depending on the orientation of $H$. Thus, in both cases, $\bis (0,a)$ and $\bis(0,b)$ are not equivalent. 

Furthermore, by Proposition~\ref{projhom}, for any facet $F$, $\mathcal{B}_{F,-F}=C_F$. In particular, if $a\in C_F$ but $b \notin C_F$ then again $\bis (0,a)$ and $\bis(0,b)$ are not equivalent. 
\end{proof}
The preceding proposition shows that the equivalence relation~\eqref{def:equivalence} on the bisectors for points in general position is finer than the partition of points in the complement of the common refinement of $\mathcal{H}_P$ and $\mathcal{F}_P$. As an immediate consequence we obtain that if the bisection fan of $P$ exists then it refines both $\mathcal{H}_P$ and $\mathcal{F}_P$ as a polyhedral fan.

    \begin{prop} \label{bis_refine}
     Let $P$ be a centrally symmetric polytope containing the origin in its interior and let $\Delta$ be the bisection fan corresponding to $\dist^{P}.$ Then both $\mathcal{H}_{P}$ and $\mathcal{F}_{P}$ are refined by $\Delta$. 
    \end{prop}
    \begin{proof}
        Follows directly from Proposition~\ref{lem:bis_refine}. \qedhere
    \end{proof}
    In general, the bisection fan properly refines the common refinement of $\mathcal{H}_{P}$ and $\mathcal{F}_{P}$. This is, for example, the case for sufficiently generic polygons as well as the discrete Wasserstein norm, see Sections~\ref{sec:polygons} and~\ref{sec:Wasserstein} below. However, there are also polytopes, such as the cross-polytope and the cube (see Sections~\ref{sec:crosspolytope} and~\ref{sec:cube}), for which the bisection fan is precisely given by the the common refinement of $\mathcal{H}_{P}$ and $\mathcal{F}_{P}$. It would be interesting to give criteria for when this phenomenon arises.

    \begin{quest}
        What are necessary and sufficient criteria for $P$ for its bisection fan to be equal to the common refinement of $\mathcal{H}_{P}$ and $\mathcal{F}_{P}$?
    \end{quest}

We conclude this section by noting that bisection cones and fans
are well-behaved under linear isomorphisms. Later on in Section~\ref{sec:polygons} we will see that
the same is not true for combinatorial equivalence: an example of combinatorially equivalent polygons for which the bisection fans are combinatorially different is discussed in Example~\ref{ex:noncombiproperty}. In conclusion, the notions of bisection cones and bisection fans are metric properties of the polytope rather than combinatorial ones.
    \begin{prop}\label{prop:lineariso}
        For any polytope $P$ and any linear isomorphism $\pi \colon \mathbb{R}^d\rightarrow \mathbb{R}^d$ we have that $\mathcal{B}_{\pi(F),\pi (G)}=\pi(\mathcal{B}_{F,G})$. In particular, if $\Delta$ is the bisection fan of $P$ then $\pi (\Delta) = \{\pi (C)\colon C\in \Delta\}$ is the bisection fan of $\pi (P)$.
    \end{prop}
    \begin{proof}
        From the description of bisection cones given in Proposition~\ref{raysbfg} it follows that for all facets $F,G$ of $P$ and linear $\pi$ we have $\pi (\mathcal{B}_{F,G})=\mathcal{B}_{\pi (F),\pi (G)}$. Furthermore, if $\pi$ is an isomorphism we also have 
        \[
        \pi (\mathbb{R}^{d}\setminus\mathcal{B}_{F,G})=\mathbb{R}^d\setminus \pi (\mathcal{B}_{F,G})=\mathbb{R}^d\setminus \mathcal{B}_{\pi (F),\pi (G)}
        \]
        In particular, it follows from Equation \eqref{e_s} that the equivalence classes of the bisectors with respect to $\pi (P)$ are precisely given by the images under $\pi$ of the equivalence classes of $P$. This proves the claim.
    \end{proof}

\section{Polygons}\label{sec:polygons}
In this section, we give an explicit description of the bisection cones and the bisection fan of centrally symmetric polygons. The set-up is as follows: let $P = \conv\{v_{1}, v_{2}, \ldots , v_{2n}\} \subset \mathbb{R}^{2}$ be a centrally symmetric $2n$-gon with vertices labelled counter-clockwise. Let $F_{i} = \conv\{v_{i}, v_{i+1}\}$ for $i=1, \ldots ,2n$ be the facets of $P$; here and throughout this section addition of indices is taken modulo $2n$.  We write $\mathcal{B}_{i, j}$ for the bisection cone $\mathcal{B}_{F_{i}, F_{j}}.$ By Proposition~\ref{dim_bfg}, $\mathcal{B}_{i, j}$ is $2$-dimensional if $i \neq j$ and otherwise $1$-dimensional. By Proposition \ref{raysbfg}, $\mathcal{B}_{i, j}$ is generated by vectors of the form $v - u$, where $v \in \{v_{i}, v_{i+1}\}$ and $u \in \{v_{j}, v_{j+1}\}.$ The following proposition provides the ray description of $\mathcal{B}_{i, j}$.

\begin{prop}\label{refined_2_dim}
Let $i, j \in [2n].$ Then
\[
\mathcal{B}_{i, j} = \begin{cases}
\mathbb{R}\cdot(v_{i+1}-v_{i}) \quad &\text{if $j =i$},\\
\cone(\{v_{i}-v_{j}, v_{i+1}-v_{j+1}\}) \quad &\text{if $i\neq j$}. \\
\end{cases}
\] 
\end{prop}

\begin{proof}

If $j=i$ then $v_{i}$ and $v_{i+1}$ are the only vertices of $F_i=F_j$ and thus $\mathcal{B}_{i, j}=\mathbb{R}\cdot(v_{i+1}-v_{i})$ by Proposition \ref{raysbfg}.

If $i\neq j$ we distinguish between the cases when $j \in \{i-1,i+1\}$ and $j\not \in \{i-1,i+1\}$. If $j=i+1$ then from Proposition \ref{raysbfg} it follows that $\mathcal{B}_{i, j}$ is the conical hull of $v_{i}-v_{i+1}, v_{i+1}-v_{i+2}, v_{i}-v_{i+2}, 0$ among which $0$ and $v_{i}-v_{i+2}$ are redundant since the latter is sum of the first and second vectors above. This shows the claim in this case. If $j= i-1$ the claim follows from the case when $j=i+1$ together with Corollary~\ref{bfg_swap}.

\begin{figure}[ht]
\begin{subfigure}[b]{0.45 \textwidth}
\centering

\begin{tikzpicture}
\node[draw=none,minimum size=4cm,regular polygon,regular polygon sides=10] (a) {};

\foreach \i in {1}
    \node[circle, label=above:$v_{\i}$] at ($(a.corner \i)+(1,0)$) {};
\foreach \i in {2}
    \node[circle, label=above:$v_{\i}$] at ($(a.corner \i)$) {};
\foreach \i in {6}
    \node[circle, label=below:$v_{\i}$] at ($(a.corner \i)+(-1,0)$) {};
\foreach \i in {7}
    \node[circle, label=below:$v_{\i}$] at (a.corner \i) {};
\foreach \i in {3,4}
    \node[circle, label=left:$v_{\i}$] at (a.corner \i) {};
\foreach \i in {5}
    \node[circle, label=left:$v_{\i}$] at ($(a.corner \i)+(-0.5,0)$) {};
\foreach \i in {8,9}
    \node[circle, label=right:$v_{\i}$] at (a.corner \i) {};
\foreach \i in {10}
    \node[circle, label=right:$v_{\i}$] at ($(a.corner \i)+(0.5, 0)$) {};

\draw ($(a.corner 1)+(1,0)$) -- (a.corner 2);
\draw [red] ($(a.corner 1)+(1,0)$) -- ($(a.corner 10)+(0.5,0)$) node[label={[xshift=0.1cm, yshift=0.3cm]$F_{10}$}]  {};
\draw [black] (a.corner 9)-- ($(a.corner 10)+(0.5,0)$) node[]  {};
\draw [red] (a.corner 4)-- ($(a.corner 5)+(-0.5,0)$) node[label={[xshift=-0.4cm, yshift=0.2cm]$F_{4}$}]  {};
\draw [black] ($(a.corner 6)+(-1,0)$)-- ($(a.corner 5)+(-0.5,0)$) node[]  {};
\draw [black] ($(a.corner 6)+(-1,0)$)-- (a.corner 7) node[]  {};

\foreach \i[evaluate=\i as \j using int(\i+1)] in {7,8}
    \draw (a.corner \i) -- (a.corner \j) {};
\foreach \i[evaluate=\i as \j using int(\i+1)] in {2,3}
    \draw (a.corner \i) -- (a.corner \j) {};

\path ($(a.corner 1)+(1,0)$)  edge (a.corner 4);
\path ($(a.corner 1)+(1,0)$) [ultra thick] edge ($(a.corner 5)+(-0.5,0)$);
\path ($(a.corner 10)+(0.5,0)$) [ultra thick] edge (a.corner 4);
\path ($(a.corner 10)+(0.5,0)$) edge ($(a.corner 5)+(-0.5,0)$);

\path[fill=red,opacity=0.1] (a.corner 4) -- ($(a.corner 1)+(1,0)$) -- ($(a.corner 10)+(0.5,0)$) -- ($(a.corner 5)+(-0.5,0)$) -- cycle;
\end{tikzpicture}
\caption{The quadrilateral formed by $F_{10}, F_{4}$.}
  \label{10gon_a_new}
\end{subfigure}
\begin{subfigure}[b]{0.45 \textwidth}
\centering
\begin{tikzpicture}
\draw[->,thick] (0,0)--(4.5,0) node[right]{$x$};
\draw[->,thick] (0,0)--(0,4.5) node[above]{$y$};
\draw [->, ultra thick] (0,0) -- ($(a.corner 10)+(0.5,0)-(a.corner 4)$) node[right]  {$v_{10} - v_{4}$};
\draw [->] (0,0) -- ($(a.corner 10)+(0.5,0)-(a.corner 5)-(-0.5,0)$) node[right]  {$v_{10} - v_{5}$};
\draw [->] (0,0) -- ($(a.corner 1)+(1,0)-(a.corner 4)$) node[right]  {$v_{1} - v_{4}$};
\draw [->, ultra thick] (0,0) -- ($(a.corner 1)+(1,0)-(a.corner 5)-(-0.5,0)$) node[right]  {$v_{1} - v_{5}$};
\path[fill=blue,opacity=0.1] (0,0) -- ($(a.corner 10)+(0.5,0)-(a.corner 4)$) -- ($(a.corner 1)+(1,0)-(a.corner 5)-(-0.5,0)$) -- cycle;

\end{tikzpicture}
\caption{The bisection cone $\mathcal{B}_{10, 4}$.}
  \label{10gon_b_new}
\end{subfigure}
\caption{The generators of $\mathcal{B}_{10, 4}$ arise from the diagonals of the quadrilateral with vertices $v_{10},v_1,v_4$ and $v_5$.}
\label{10gon_new}
\end{figure}
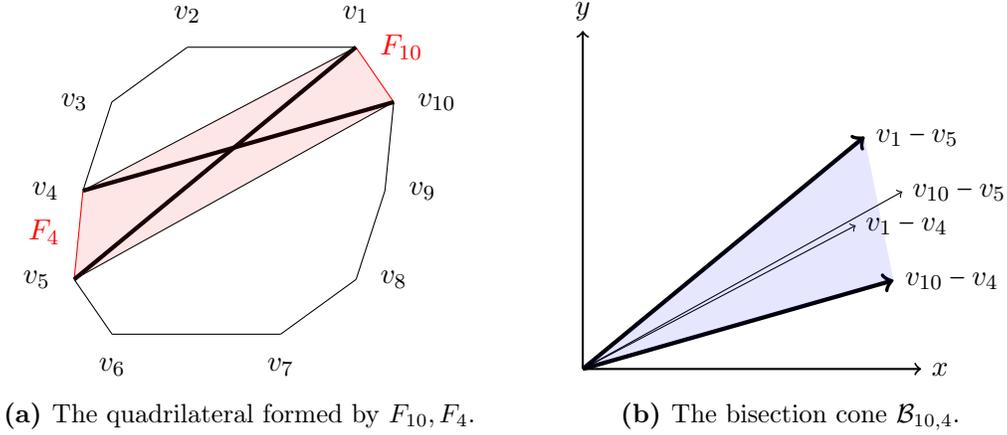

If $j\not \in \{i-1,i+1\}$ we observe that $\mathcal{B}_{i, j}$ is generated by those two $v-u$, $v\in \{v_i,v_{i+1}\}$ and $u\in \{v_j,v_{j+1}\}$, such that the angle between the corresponding line through $u$ and $v$ and the positive $x$-axis is the greatest and least, respectively. 
This is equivalent to finding the diagonals in the quadrilateral with vertices $v_{i}, v_{i+1}, v_{j}, v_{j+1}$, listed in counterclockwise order, which are precisely $\{v_{i}-v_{j}, v_{i+1}-v_{j+1}\}$, see Figure \ref{10gon_new} for an illustration.
\end{proof}

In order to determine the bisection fan of polygons we need the following decomposition lemma for half-planes. For any distinct $i,j$ in $[2n]$ let $H_{i,j}^\geq$ be the half-plane consisting of all vectors that enclose an angle between $0$ and $\pi$ with $v_i-v_j$ in counter-clockwise direction.

Then we have the following decomposition of $H_{i,j}^\geq$ into bisection cones the interiors of which are pairwise disjoint.

\begin{lem}\label{lem:2dim_biscone_decomp}
    For distinct $i, j \in [2n],$ 
    \begin{equation*}\label{eq:hs_decomp_poly}
 H_{i,j}^{\geq} = \bigcup_{k=0}^{n-1}\mathcal{B}_{i+k, j+k}  
\end{equation*}
where 
\begin{align*}
    \mathcal{B}_{i+k, j+k}\cap \mathcal{B}_{i+k+1, j+k+1}&=\mathbb{R}\cdot (v_{i+k+1}-v_{j+k+1}) &&\text{ for }0\leq k\leq n-2 \, , \text{ and}\\
    \interior \mathcal{B}_{i+k, j+k} \cap \interior \mathcal{B}_{i+\ell, j+\ell}&=\emptyset &&\text{ for  }0\leq k\neq \ell\leq n-1 \, .
\end{align*}
\end{lem}

\begin{proof}
     Observe that if we start at $v_{i}-v_{j}$ and move in a counter-clockwise direction, we encounter $v_{i+1}-v_{j+1}, v_{i+2}-v_{j+2}, \ldots , v_{i+n}-v_{j+n}$ in that order. Further,  by central symmetry of $P$, $v_{i+n}-v_{j+n} = v_{j}-v_{i}=-(v_i-v_j)$. In particular, for $k=1, \ldots , n-1$ each $v_{i+k}-v_{j+k}$ lies in $H^{\geq}_{i,j}$. The claim now follows from Proposition~\ref{refined_2_dim}, since for $k=0, \ldots, n-1$, each bisection cone $\mathcal{B}_{i+k, j+k}$ is generated by the consecutive vectors $v_{i+k}-v_{j+k},v_{i+k+1}-v_{j+k+1}$. See Figure~\ref{decagon_2} for an illustration of the decomposition. 
\end{proof}

 \begin{figure}[ht]
\centering
\scalebox{0.80}{
\begin{tikzpicture}
    \node[draw=none,minimum size=4cm,regular polygon,regular polygon sides=10] (a) {};

\foreach \i in {1}
    \node[circle, label=above:$v_{\i}$] at ($(a.corner \i)+(1,0)$) {};
\foreach \i in {2}
    \node[circle, label=above:$v_{\i}$] at ($(a.corner \i)$) {};
\foreach \i in {6}
    \node[circle, label=left:$v_{\i}$] at ($(a.corner \i)+(-1,0)$) {};
\foreach \i in {7}
        \node[circle, label=right:$v_{\i}$] at (a.corner \i) {};
\foreach \i in {3,4}
    \node[circle, label=left:$v_{\i}$] at (a.corner \i) {};
\foreach \i in {5}
    \node[circle, label=left:$v_{\i}$] at ($(a.corner \i)+(-0.5,0)$) {};
\foreach \i in {8,9}
    \node[circle, label=right:$v_{\i}$] at (a.corner \i) {};
\foreach \i in {10}
    \node[circle, label=right:$v_{\i}$] at ($(a.corner \i)+(0.5, 0)$) {};

\draw ($(a.corner 1)+(1,0)$) -- (a.corner 2);
\draw [black] ($(a.corner 1)+(1,0)$) -- ($(a.corner 10)+(0.5,0)$) node[]  {};
\draw [black] (a.corner 9)-- ($(a.corner 10)+(0.5,0)$) node[]  {};
\draw [black] (a.corner 4)-- ($(a.corner 5)+(-0.5,0)$) node[]  {};
\draw [black] ($(a.corner 6)+(-1,0)$)-- ($(a.corner 5)+(-0.5,0)$) node[]  {};
\draw [black] ($(a.corner 6)+(-1,0)$)-- (a.corner 7) node[]  {};

\foreach \i[evaluate=\i as \j using int(\i+1)] in {7,8}
    \draw (a.corner \i) -- (a.corner \j) {};
\foreach \i[evaluate=\i as \j using int(\i+1)] in {2,3}
    \draw (a.corner \i) -- (a.corner \j) {};
\foreach \i[evaluate=\i as \j using {int(Mod(\i+6,10))}] in {2,3}
    \draw [->, ultra thick] (a.center) -- ($(a.corner \i)-(a.corner \j)$) node[left]  {$v_{\i}-v_{\j}$};
    \draw [->, ultra thick] (a.center) -- ($(a.corner 7)-(a.corner 1)-(1,0)$) node[left, label=below:$v_{6} - v_{2}$, label={}, label={[xshift=-0.9cm, yshift=0.3cm]$\mathcal{B}_{5, 1}$}]{};
\draw [->, ultra thick] (a.center) -- ($(a.corner 1)+(1,0)-(a.corner 7)$) node[right, label={[xshift=-1.9cm, yshift=-1cm]$\mathcal{B}_{1, 7}$}]{$v_{1} - v_{7}$};
\draw [->, ultra thick] (a.center) -- ($(a.corner 4)-(a.corner 10)-(0.5, 0)$) node[right, label = {[xshift=-0.9cm, yshift=-0.4cm]$v_{4} - v_{10}$}, label={[xshift=0.6cm, yshift=0.6cm]$\mathcal{B}_{3, 9}$}, label={[xshift=1.4cm, yshift=2.5cm]$\mathcal{B}_{2, 8}$}]{};
\draw [->, ultra thick] (a.center) -- ($(a.corner 5)+(-1, 0)-(a.corner 1)-(1,0)$) node[right, label={[xshift=-0.9cm, yshift=-0.6cm]$v_{5} - v_{1}$}, label={[xshift=0cm, yshift=0.45cm]$\mathcal{B}_{4, 10}$}, label={[xshift=-1.50cm, yshift=2.50cm]$\mathbf{H_{1, 7}^{\geq}}$}]{};
\end{tikzpicture}
}
\caption{The half-plane $H^{\geq}_{1,7}$ is subdivided into $5$ bisection cones $\mathcal{B}_{1, 7}$, $\mathcal{B}_{2, 8}$, $\mathcal{B}_{3, 9}$, $\mathcal{B}_{4, 10}$ and $ \mathcal{B}_{5, 1}.$}
\label{decagon_2}
\end{figure}
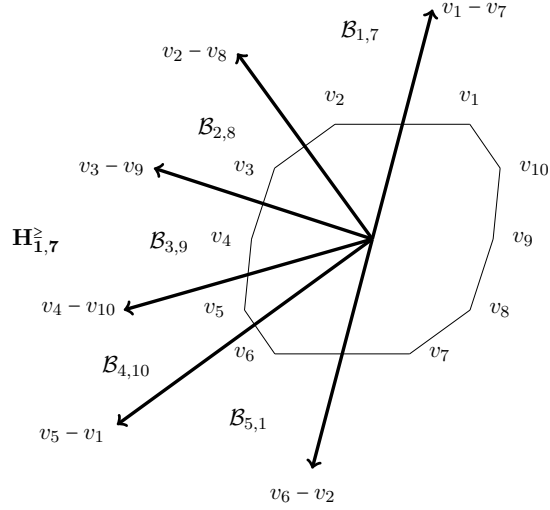

The decomposition provided by the above lemma allows us to determine the bisection fan of centrally symmetric polygons.

\begin{thm}\label{prop:bisfan_polygon}
Let $P \subset \mathbb{R}^{2}$ be a centrally symmetric $2n$-gon. Then the bisection fan of $P$ is the complete polyhedral fan $\Delta$ the rays of which are given by $v_{i}-v_{j}$, for every $i, j \in [2n]$, $i\neq j$.
\end{thm}

\begin{proof}
We need to show the following: \begin{enumerate}
    \item For two points $a, b \in \mathbb{R}^{2}$ lying in the interior of the same maximal cone of $\Delta$, $\bis^{P}(0, a)$ is equivalent to $\bis^{P}(0, b)$, and 
    \item For two points $a, b \in \mathbb{R}^{2}$ lying in the interior of different maximal cones of $\Delta$, $\bis^{P}(0, a)$ is not equivalent to $\bis^{P}(0, b).$
\end{enumerate}
For the first point, assume that $a$ and $b$ lie in the interior of the same maximal cone but $\bis^{P}(0, a)$ is not equivalent to $\bis^{P}(0, b).$ Then we can find some bisection cone $\mathcal{B}_{i, j}$ containing $a$ but not $b$. Then, by Proposition~\ref{refined_2_dim}, there must be a line generated by $v-w$, for $v \in \{v_{i}, v_{i+1}\}$ and  $w \in \{v_{j}, v_{j+1}\}$ that separates $a$ and $b$. By the definition of $\Delta$, this contradicts the assumption on $a$ and $b.$

For the second point, let $a$ and $b$ be two points lying in the interior of different maximal cones of $\Delta.$  By definition of $\Delta$, there exists a line $H$ separating $a$ and $b$ of the form $H = \mathbb{R} \cdot (v_{i} - v_{j})$, for some $1 \leq i \neq j \leq d.$ Without loss of generality we may assume that $a\in H^{\geq}_{i,j}$ and $b\not \in H^{\geq}_{i,j}$. Then, by Lemma~\ref{lem:2dim_biscone_decomp}, $a\in \mathcal{B}_{i+k,j+k}$ but $b\not \in \mathcal{B}_{i+k,j+k}$ for some $0\leq k \leq n-1$. Therefore, $\bis^{P}(0, a)$ is not equivalent to $\bis^{P}(0, b)$, as claimed. 
\end{proof}

In the previous section, we emphasised that the concept of the bisection fan was a metric rather a combinatorial one. Here we justify the latter claim by presenting an example of combinatorially isomorphic polygons with non-isomorphic bisection fans.

\begin{example}\label{ex:noncombiproperty}
Consider the regular hexagon with vertices $w_{1}, \ldots , w_{6}$ and an irregular hexagon whose vertices $v_{1}, \ldots , v_{6}$ have been perturbed, for example, in the following manner: $v_{1} = w_{1}, v_{2} = w_{2}, v_{3} = w_{3} + (0.4, -0.4), v_{4} = w_{4}, v_{5} = w_{5}, v_{6} = w_{6} +(-0.4, 0.4).$ Observe that no  main diagonal is parallel to any edge of the hexagon in the second case, while there are several such examples in the first case. Thus, the number of rays of the bisection fan in of the irregular hexagon is greater than that of the regular one. See Figure~\ref{10gon} for an illustration. 

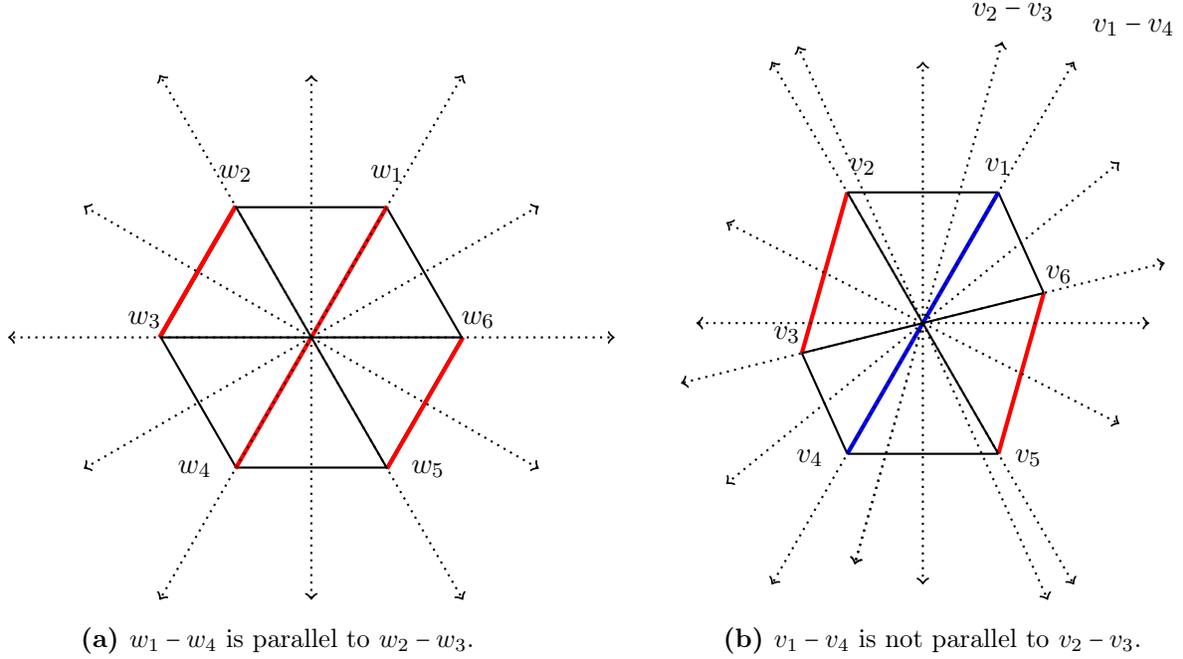
\begin{figure}[h]
\begin{subfigure}[b]{0.45\textwidth}
\centering
   \begin{tikzpicture}
   \node [
    regular polygon, 
    regular polygon sides=6, 
    minimum size=4 cm, 
    draw=black, thick
] 
at (0,0) (A) {};
\foreach \i in {1,2}
    \node[circle, label=above:$w_{\i}$] at (A.corner \i) {};
\foreach \i in {4}
    \node[circle, label=left:$w_{\i}$] at (A.corner \i) {};
\foreach \i in {3}
    \node[circle, label={[xshift=-0.2 cm, yshift= -0.2cm]{$w_{\i}$}}] at (A.corner \i) {};
\foreach \i in {5}
    \node[circle, label=right:$w_{\i}$] at (A.corner \i) {};
\foreach \i in {6}
    \node[circle, label={[xshift=0.2 cm, yshift= -0.2cm]{$w_{\i}$}}] at (A.corner \i) {};
\draw [red, ultra thick] ($(A.corner 2)$)  -- ($(A.corner 3)$) node[right, label ={}]{}; 
\draw [red, ultra thick] ($(A.corner 2)$)  -- ($(A.corner 3)$) node[right, label ={}]{}; 
\draw [red, ultra thick] ($(A.corner 1)$)  -- ($(A.corner 4)$) node[right, label ={}]{}; 
\draw [red, ultra thick] ($(A.corner 5)$)  -- ($(A.corner 6)$) node[right, label ={}]{}; 

\draw [black, thick] ($(A.corner 2)$)  -- ($(A.corner 5)$) node[right, label ={}]{};
\draw [black, thick] ($(A.corner 3)$)  -- ($(A.corner 6)$) node[right, label ={}]{};
\foreach \i in {1,...,6}
    \draw [dotted, thick, ->] (A.center) -- ($2*(A.corner \i)$) node[right]  {};
 \draw [dotted, thick, ->] (A.center) -- ($(A.corner 1)+(A.corner 6)$) node[right]  {};
\foreach [
    evaluate=\i as \j using int(\i+1)
    ] \i in {1,...,5}
    \draw [dotted, thick, ->] (A.center) -- ($(A.corner \i)+(A.corner \j)$) node[right, label=above:{}]  {};
\end{tikzpicture}
\caption{$w_{1}-w_{4}$ is parallel to $w_{2}-w_{3}$.}
  \label{10gon_a}
\end{subfigure}
~
\hspace{1 cm}
\begin{subfigure}[b]{0.45 \textwidth}
\centering
\begin{tikzpicture}
\node[draw=none,minimum size=4cm,regular polygon,regular polygon sides=6] (a) {};

\node[dot=0pt, fill=black, label={[xshift=0cm, yshift=0.1cm]{$v_{1}$}}] at ($(a.corner 1)$) (v1) {};
\node[dot=0pt, fill=black, label={[xshift=0.2cm, yshift=0.1cm]{$v_{2}$}}] at ($(a.corner 2)$)  (v2) {};
\node[dot=0pt, fill=black, label={[xshift=-0.2cm, yshift=0cm]{$v_{3}$}}] at ($(a.corner 3)+(0.4, -0.4)$) (v3) {};
\node[dot=0pt, fill=black, label={[xshift=-0.5cm, yshift=-0.3cm]{$v_{4}$}}] at ($(a.corner 4)$) (v4) {};
\node[dot=0pt, fill=black, label={[xshift=0.4cm, yshift=-0.3cm]{$v_{5}$}}] at ($(a.corner 5)$) (v5) {};
\node[dot=0pt, fill=black, label={[xshift=0.2cm, yshift=0cm]{$v_{6}$}}] at ($(a.corner 6)+(-0.4, 0.4)$) (v6) {};

\draw [red, ultra thick] (v2)  -- (v3) node[right, label ={}]{};
\draw [red, ultra thick] (v5)  -- (v6) node[right, label ={}]{};

\foreach \i[evaluate=\i as \j using int(\i+1)] in {1,3,4}
    \draw [black, thick] (v\i) -- (v\j) node[right]  {};
\draw [black, thick] (v1) -- (v6) node[right]  {};

\draw [blue, ultra thick] (v1)  -- (v4) node[right, label ={}]{};
\draw [black, thick] (v2)  -- (v5) node[right, label ={}]{};
\draw [black, thick] (v3)  -- (v6) node[right, label ={}]{};
\foreach \i in {1,...,6}
    \draw [dotted, thick, ->] (a.center) -- ($2*(v\i)$) node[right]  {};
 \draw [dotted, thick, ->] (a.center) -- ($(v1)+(v6)$) node[right, label=above:{}]  {};
 \draw [dotted, thick, ->] (a.center) -- ($1.75*(v2)-1.75*(v3)$) node[right, label=above:{$v_{2}-v_{3}$}, label={[xshift=1.6cm, yshift= -0.2cm]{$v_{1}-v_{4}$}}]  {};
 \draw [dotted, thick, ->] (a.center) -- ($1.5*(v4)-1.5*(v5)$) node[right]  {};
 \draw [dotted, thick, ->] (a.center) -- ($1.5*(v5)-1.5*(v4)$) node[right]  {};
  \draw [dotted, thick, ->] (a.center) -- ($1.5*(v3)-1.5*(v2)$) node[right]  {};
   \draw [dotted, thick, ->] (a.center) -- ($1.5*(v5)-1.5*(v6)$) node[right]  {};
    \draw [dotted, thick, ->] (a.center) -- ($2.75*(v1)-2.75*(v6)$) node[right]  {};
    \draw [dotted, thick, ->] (a.center) -- ($2.75*(v6)-2.75*(v1)$) node[right]  {};

\foreach [
    evaluate=\i as \j using int(\i+1)
    ] \i in {1,...,5}
    \draw [dotted, thick, ->] (a.center) -- ($(v\i)+(v\j)$) node[right, label=above:{}]  {};
\end{tikzpicture}
\caption{$v_{1}-v_{4}$ is not parallel to $v_{2}-v_{3}$.}
  \label{10gon_b}
\end{subfigure}
\caption{The main diagonals of a regular hexagon are parallel to two edges whereas that does not hold in the case of the irregular hexagon defined above. Hence as indicated, the bisection fan of the irregular hexagon has more rays than that of the regular one, rendering them non-isomorphic as fans.} 
\label{10gon}
\end{figure}

\end{example}

\section{\texorpdfstring{$\ell_{\infty}$}{linf}-norm}\label{sec:cube}

In this section, we determine the bisection fan of the $d$-dimensional cube $P = [-1, +1]^{d}.$ Its distance function is the familiar infinity norm \[
\dist \nolimits^{P}(0, x)= || x||_{\infty}= \max\{|x_{i}|: i=1, \ldots , d\}.
\]
The facets of $P$ are of the form $F_{i}^{+}= \{x \in P: x_{i}=1\}$ and $F_{i}^{-}= \{x \in P: x_{i}=-1\}$, for $i=1, \ldots , d.$ The  corresponding facet cones are:
\begin{align*}
        C_{F_{i}^{+}}&= \{x \in \mathbb{R}^{d}: x_{i}\pm x_{k}\geq 0, k \neq i\} \quad \text{for $i=1, \ldots, d$, and }\\ 
        C_{F_{i}^{-}}&= \{x \in \mathbb{R}^{d}: -x_{i}\pm x_{k}\geq 0, k \neq i\} \quad \text{for $i=1, \ldots, d$}.
    \end{align*}
We now turn to determining the bisection cones of the cube. For convenience, we fix the following notation for the bisection cones and coordinate half-spaces respectively: \begin{align*}
   \mathcal{B}_{{i}^{\sigma_{1}}, \, {j}^{\sigma_{2}}}  & \coloneqq \mathcal{B}_{F_{i}^{\sigma_{1}}, \,  F_{j}^{\sigma_{2}}}  \quad \text{for $\sigma_{1}, \sigma_{2} \in \{+,-\}$}, \\ \mathcal{H}_{i} & \coloneqq \{x \in \mathbb{R}^{d}: x_{i} \geq 0\} \quad \text{for $i \in [d]$}.
\end{align*}

The bisection cones of the cube can then be succinctly described as the intersection of the following  coordinate half-spaces.

\begin{prop}\label{cube:bis_cone}
    Let $\sigma_{1}, \sigma_{2} \in \{+, -\}.$ When $1 \leq i \neq j \leq d$, we have
   \begin{align*}
\mathcal{B}_{i^{\sigma_{1}}, \, j^{\sigma_{2}}} &= \sigma _1\mathcal{H}_{i} \cap (-\sigma_2\mathcal{H}_{j})
\\ \hspace{-2.2 cm}
\text{ \hspace{-2.2 cm}When $i=j$, the bisection cones are:}\\
\mathcal{B}_{i^{\sigma_{1}}, \, i^{\sigma_{2}}} &= \begin{cases}
\{a \in \mathbb{R}^{d}: a_{i} = 0\} \quad &\text{if $\sigma_{1}=\sigma_{2}$}\\
C_{F_{i}^{\sigma_{1}}} \quad &\text{if $\sigma_{1} \neq \sigma_{2}$.}
\end{cases}
\end{align*}
\end{prop}
\begin{proof}
We begin by remarking that the cases when $i=j$ follow immediately from the second half of Proposition \ref{projhom}. For the remaining cases, let $1 \leq i \neq j \leq d$. 

We begin with the case $(\sigma _1,\sigma _2)=(+,+)$. Let $a \in \mathcal{B}_{i^{+}, \, j^{+}}.$ Then there exists $x \in \bis(0, a) \cap C_{F_{i}^{+}} \cap ( C_{F_{j}^{+}} +a).$ This is equivalent to the fact that the following linear inequalities hold:
\begin{align}
    &x_{i} = x_{j} - a_{j} \label{cube_1}  \\
    &x_{i} + x_{k} \geq 0 \label{cube_2} \quad \text{for $k \neq i$} \\
    &x_{i} - x_{k} \geq 0 \label{cube_3} \quad \text{for $k \neq i$} \\
    &x_{j} -a_{j} + x_{l} -a _{l} \geq 0 \quad \text{for $l \neq j$}\label{cube_4} \\
    &x_{j} -a_{j} - x_{l} + a _{l} \geq 0 \label{cube_5} \quad \text{for $l \neq j$}
\end{align}
Combining (\ref{cube_1}) and (\ref{cube_5}) with $l=i$ yields that $a_{i} \geq 0$. Similarly, combining (\ref{cube_1}) and (\ref{cube_3}) with $k=j$ yields that $a_{j} \leq 0$. Together, we get that $a \in \mathcal{H}_{i} \cap (-\mathcal{H}_{j}).$ To show the opposite direction, let $a \in \mathcal{H}_{i} \cap (-\mathcal{H}_{j})$ and define $x \in \mathbb{R}^{d}$ as:
\[x_{r} = \begin{cases}
\sum_{m=1}^{d}|a_{m}| \quad &\text{if $r=i$}\\
\sum_{m=1}^{d}|a_{m}| + a_{j} \quad &\text{if $r=j$}\\
a_{r} \quad &\text{if $r \neq i,j$}
\end{cases}
\]
We claim that this point $x$ lies in $\bis_{F_{i}^{+}, \, F_{j}^{+}}(0, a).$ To this end, we need to establish the five conditions above. Equation~(\ref{cube_1}) holds immediately. Inequalities~(\ref{cube_2}) and (\ref{cube_3}) are equivalent to the inequality $x_{i} \geq |x_{k}|$ and this holds immediately for $k \neq i, j.$ It also holds for $k = j$ since \[
x_{i} = \sum_{m=1}^{d}|a_{m}| \geq \sum_{m=1}^{d}|a_{m}| +a_{j}= |x_{j}|.
\]
Similarly inequalities~(\ref{cube_4}) and (\ref{cube_5}) are equivalent to $x_{j}-a_{j} \geq |x _{l} - a_{l}|$ and this holds vacuously when $l \neq i, j$ since the right hand side is zero in this case. This inequality also holds when $l=i$ since \[
x_{j}-a_{j} = x_{i} = \sum_{m=1}^{d}|a_{m}| \geq \sum_{m=1}^{d}|a_{m}| - a_{i}= |x_{i} -a_{i}|.
\] 
With the five inequalities (\ref{cube_1}) to (\ref{cube_5}) being satisfied, we can conclude that $x$ lies in $\bis_{F_{i}^{+}, \, F_{j}^{+}}(0, a).$ This completes the proof of the first case of the proposition.

The remaining three cases, $(\sigma _1,\sigma _2)\in \{(+,-),(-,+), (-,-)\}$, follow from the first case by symmetry. More precisely, the case $(\sigma _1,\sigma _2)=(+,-)$ follows from Proposition~\ref{prop:lineariso} by considering the isomorphism $\pi$ that flips the sign of the $j$-th coordinate. The third case -- that is the case when $\sigma_{1}$ and $\sigma_{2}$ are $-$ and $+$ respectively -- follows from the second case and the third equation in Corollary \ref{bfg_swap}, namely that $\mathcal{B}_{-F, -G}= -\mathcal{B}_{F, G}$. Similarly the case when $\sigma_{1} = -1, \sigma_{2}=-1$ follows from the first case and  Corollary \ref{bfg_swap}.
\end{proof}

The next proposition explicitly describes the bisection fan of the $d$-cube. It is the common refinement of the face fan $\mathcal{F}_{P}$ and the coordinate hyperplane arrangement $\mathcal{H}.$ See Figure \ref{fig:bisfan_cube} for a picture of the $3$-dimensional case.

\begin{thm}\label{prop:bisfancube}
Let $P \subset \mathbb{R}^{d}$ be the cube $[-1, 1]^{d}.$ The bisection fan of $P$ exists and is equal to the common refinement of the fan $\mathcal{H}$ induced by the coordinate hyperplanes and the face fan $\mathcal{F}_{P}$ of~$P$.
\end{thm}
\begin{proof}
Let $\Delta$ be the common refinement of $\mathcal{H}$ and $\mathcal{F}_{P}.$
We need to show that for any $a, b \in \mathbb{R}^{d}$ in general position, $a, b$ are in the interior of the same maximal cone of $\Delta$ if and only if $\bis^{P}(0, a)$ is equivalent to $\bis^{P}(0, b)$; the latter is equivalent to showing that $a$ and $b$ are simultaneously contained or not contained in $\mathcal{B}_{i^{\sigma _1},j^{\sigma _2}}$ for any choices of $i,j \in [d]$ and $\sigma_1,\sigma _2 \in \{+, -\}$. \\
Any maximal cone in $\Delta$ is uniquely given as the intersection of a maximal cone of $\mathcal{H}$ and a maximal cone $\mathcal{F}_{P}$, that is, a facet cone of $P$. Let $a\in C_1\cap C_{F_{i_1}^{\pi_1}}$ and $b\in C_2\cap C_{F_{i_2}^{\pi_2}}$ where $C_1,C_2$ are maximal cones of $\mathcal{H}$ and $C_{F_{i_1}^{\pi_1}}, C_{F_{i_2}^{\pi_2}}$ are facet cones. Since $C_1$ and $C_2$ are orthants, $a$ and $b$ have the same positive and negative support if and only if $C_1=C_2$. On the other hand, by Proposition~\ref{cube:bis_cone}, this is the case if and only if $a$ and $b$ are simultaneously contained or not contained in $\mathcal{B}_{i^{\sigma _1},j^{\sigma _2}}$ for all $i\neq j$ and arbitrary $\sigma _1,\sigma _2\in \{+,-\}$. Furthermore, for all $i=j$, $\sigma _1\neq \sigma _j$, we have that $a$ and $b$ are simultaneously contained in $\mathcal{B}_{i^{\sigma _1},i^{\sigma _2}}=C_{F_i^{\sigma _1}}$ if and only if $C_{F_i^{\sigma _1}}= C_{F_{i_1}^{\pi _1}}=C_{F_{i_2}^{\pi _2}}$. Lastly, we observe that, if $i=j$ and $\sigma_1=\sigma _2$, then $\mathcal{B}_{i^{\sigma _1},j^{\sigma _2}}$ is a hyperplane and thus neither $a$ nor $b$ are contained in $\mathcal{B}_{i^{\sigma _1},j^{\sigma _2}}$ in this case as $a,b$ are in general position. In summary, in all cases, $a$ and $b$ are simultaneously contained or not contained in $\mathcal{B}_{i^{\sigma _1},j^{\sigma _2}}$ if and only if $C_1=C_2$ and $C_{F_{i_1}^{\pi _1}}=C_{F_{i_2}^{\pi _2}}$, that is, if and only if $a$ and $b$ are contained in the same maximal cone of $\Delta$. This proves the claim.\qedhere
\end{proof}

\begin{figure}[ht]
    \centering
    \includegraphics[scale=0.5]{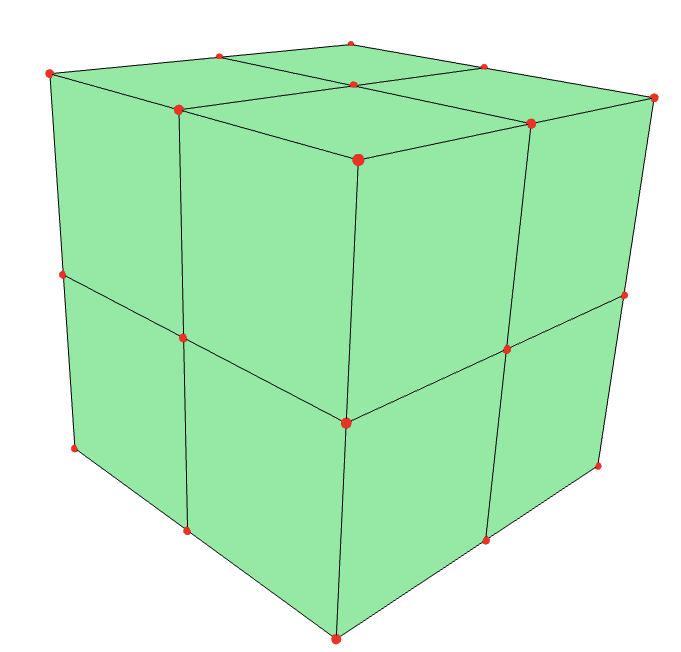}
    \caption{The cube $[-1, 1]^{3}$ intersected with its bisection fan.}
    \label{fig:bisfan_cube}
\end{figure}

    \section{\texorpdfstring{$\ell_{1}$}{l1}-norm}\label{sec:crosspolytope}
    In this section we give an explicit description of the bisection cones and the bisection fan of the $\ell _{1}$-norm
    \[
    \|x\|_1 \ = \ |x_1|+|x_2|+\cdots +|x_d| \, .
    \]
    
    The corresponding unit ball is the $d$-dimensional cross-polytope 
    \[
    P \ = \ \conv\{\pm \, e_{i}: i \in [d]\} \, 
    \]
    with facets
    \[
    F_I \ = \ \conv\left(\{e_i\colon i\in I\}\cup \{-e_i \colon i\in I^c\}\right) \, , \quad I\subseteq [d] \,,
    \]
    where $e_1,\ldots, e_d$ denote the standard basis vectors in $\mathbb{R}^d$ and $I^c$ the complement $[d]\setminus I$.  The corresponding facet cones are given by the orthants
 \[
\mathcal{O}_I \ = \ \left\{x \in \mathbb{R}^{d}: x_{i} \geq 0 \, \, \text{for} \, \, i \in I  \, \,\text{and} \, \, x_{i} \leq 0 \, \, \text{for} \, \, i \in I^{c}\right\} \, .
 \]
 In particular, the face fan of $P$ is induced by the coordinate hyperplanes. With the notation 
\begin{equation*}
    x(I) \coloneqq \sum_{i \in I}x_{i}\,,
\end{equation*}
for all $x\in \mathcal{O}_I$ we have 
\[
 \|x\| _1 \ = \ x(I)-x(I^{c}) \, .
 \]

    We will also use the notation $\bis_{I, J}(0, a)$ for the cell $\bis_{F_{I}, F_{J}}(0, a)$, and $\mathcal{B}_{I,J}$ for the bisection cone $\mathcal{B}_{F_I,F_J} = \{a \in \mathbb{R}^{d}: \bis_{I, J}(0, a)\cap \mathcal{O}_{I} \cap (\mathcal{O}_{J}+a) \neq \emptyset\}$. The following proposition provides an inequality description of $\mathcal{B}_{I,J}$.
    
    \begin{prop}\label{bis_tree}
    For any subsets $I,J\subseteq [d]$ the bisection cone $\mathcal{B}_{I,J}$ is given by the following set of inequalities.
    \begin{itemize}
        \item[(i)] $a_{i} \geq 0 \quad \text{ for all } \, i \in I \cap J^{c} \, .$
        \item[(ii)]$a_{i} \leq 0 \quad \text{ for all }  \, i \in J \cap I^{c} \, .$
       \item[(iii)] $a(J) \leq a(J^{c}).$
    \item[(iv)]$a(I^{c}) \leq a(I).$
    \end{itemize}

    \end{prop}
    \begin{proof}
    We first show that any element $a\in \mathcal{B}_{I,J}$ satisfies the above inequalities. For any such $a\in \mathcal{B}_{I,J}$ the intersection $\bis(0,a)\cap \mathcal{O}_I\cap (\mathcal{O}_J+a)$ is non-empty. Thus, by definition, there exists some $x\in \mathbb{R}^d$ such that
    \begin{itemize}
        \item[(a)] $\|x-a\|_1=\|x\|_1$\, ,
        \item[(b)] $x_i\geq 0$ for all  $i\in I$ and $x_i\leq 0$ for all  $i\in I^c$, and
        \item[(c)] $x_i\geq a_i$ for all $i\in J$ and $x_i\leq a_i$  for all $i\in J^c$ \, .
    \end{itemize}
    From (b) and (c) it follows that
    \begin{eqnarray}
   && a_i\geq  x_i \geq  0 \text{ for all }i\in I\cap J^c \, ,\label{eq:1} \\ 
    &&a_i\leq x_i \leq 0 \text{ for all }i\in I^c\cap J\, .\label{eq:2}
    \end{eqnarray}
    In particular, conditions (i) and (ii) are satisfied by $a$. Moreover, we obtain
    \begin{equation}
    a(J \cap I^{c}) - a(I \cap J^{c}) \leq x(J \cap I^{c}) - x(I \cap J^{c}) \leq \ 0 \label{eq:3}
    \end{equation}
    by summing over the relevant ranges for both these sets of inequalities~\eqref{eq:1} and~\eqref{eq:2} and then subtracting the first from the second sum.

    Further, we observe that for all $x$ that satisfy (b) and (c), condition (a) is equivalent to
    \[
    (x-a)(J) - (x-a)(J^{c}) = x(I) - x(I^{c})
    \]
    which can be simplified to 
    \begin{equation}
     2 \left(x(J \cap I^{c}) - x(I \cap J^{c}) \right) \ = a(J) - a(J^{c}) . \label{eq:4}
    \end{equation}
    We observe that the left hand side of~\eqref{eq:3}, multiplied with $2$, can be written as
    \begin{equation*}
   ( a(J)-a(J \cap I) - (a(I)- a(I \cap J))) + (a(I^{c})- a(I^{c} \cap J^{c}) - (a(J^{c}) - a(J^{c} \cap I^{c})))
    \end{equation*}
    which simplifies to
    \begin{equation}
    a(J)-a(J^{c}) + a(I^{c}) - a(I). \label{eq:5}
    \end{equation}
    Together with~\eqref{eq:3} and~\eqref{eq:4} the expression in~\eqref{eq:5} yields
    \[
     a(J)-a(J^{c}) + a(I^{c}) - a(I) \leq a(J)-a(J^{c}) \leq 0.
    \]
    This shows that $a$ also satisfies properties (iii) and (iv).
    
    In order to show that the inequalities (i)-(iv) completely describe the bisection cone $\mathcal{B}_{I,J}$, we will prove that for every point $a$ that satisfies these conditions, we can construct a point $x\in \bis ^{P}(0,a)\cap \mathcal{O}_I\cap (\mathcal{O}_J+a)$.
    
    For every such point $a$ we consider the expression
    \[
    D \ = 2( a(J \cap I^{c}) - a(I \cap J^{c})) = a(J) - a(J^{c})+a(I^{c})-a(I)
    \]
    By conditions (iii) and (iv) we have that $D\leq 0$. We distinguish between the cases $D=0$ and $D<0$.
    
    If $D<0$ we consider
    \[
    \lambda = \frac{1}{D}\left(a(J)-a(J^{c})\right)
    \]
    and define the point $x$ by setting
    
    \[
    x_{i} =  \begin{cases} 
\lambda a_{i}, & i \in (J \cap I^{c}) \cup (I \cap J^{c}) \\ \max\{0, a_{i}\}, & i \in I \cap J \\ \min\{0, a_{i}\} & i \in I^{c} \cap J^{c}
            \end{cases} \, .
\]
By conditions (iii) and (iv), $\lambda$ is contained in the interval $[0,1]$. From that we observe that $x$ is contained in $\mathcal{O}_I\cap (\mathcal{O}_J+a)$. Moreover,
\[x(J \cap I^{c}) - x(I \cap J^{c}) = \lambda \left( a(J \cap I^{c})- a(I \cap J^{c})\right) = \dfrac{1}{2} \left( a(J) - a(J^{c})\right)
\]
which is equivalent to $\|x-a\|_1=\|x\|_1$ for all $x\in \mathcal{O}_I\cap (\mathcal{O}_J+a)$ by equation~\eqref{eq:4} above. Thus, $x$ is contained in $\bis ^{P}(0,a)\cap \mathcal{O}_I\cap (\mathcal{O}_J+a)$.

In case $D=0$ we define $x\in \mathbb{R}^d$ by setting
    \[
    x_{i} =  \begin{cases} 
0, & i \in (J \cap I^{c}) \cup (I \cap J^{c}) \\ \max\{0, a_{i}\}, & i \in I \cap J \\ \min\{0, a_{i}\} & i \in I^{c} \cap J^{c}
            \end{cases} \, .
\]
Again, we observe that $x$ is contained in $\mathcal{O}_I\cap (\mathcal{O}_J+a)$. Further, by conditions (i) and (ii) and $D=0$, we have that $a_i=0$ for all $i\in (J \cap I^{c}) \cup (I \cap J^{c})$. Thus, again equation~\eqref{eq:4} is satisfied and $x$ lies in $\bis ^{P}(0,a)\cap \mathcal{O}_I\cap (\mathcal{O}_J+a)$. This completes the proof.
\end{proof}

Proposition \ref{bis_tree} allows us to determine the bisection fan of the $\ell_{1}$-norm. It is given by the hyperplane arrangement consisting of the hyperplanes parallel to the facets of the cross-polytope together with the coordinate hyperplanes which induce the face fan. For an illustration of the $2$- and $3$-dimensional bisection fans of the $\ell_{1}$-norm, see Figure~\ref{fig:bis_fan_l1_norm}.
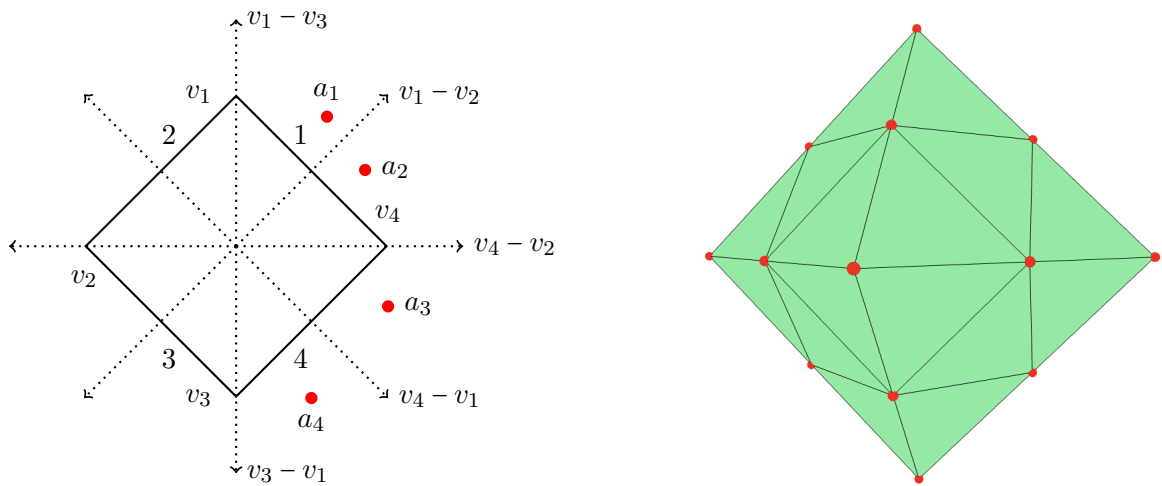
\begin{figure}[ht!]
\begin{subfigure}[b]{0.45\textwidth}
\centering
 \begin{tikzpicture}
    \node [
    regular polygon, 
    regular polygon sides=4, 
    minimum size=4 cm, shape border rotate = 45, scale = 1,
    draw=black, thick
] 
at (0,0) (A) {};
\draw node[fill=red,circle,scale=0.45, label = $a_{1}$] at ($(A)+(15:1.25)+(0, 1.4)$) {};
\draw node[fill=red,circle,scale=0.45, label =right: $a_{2}$] at ($(A)+(45:2.0)+(0.3, -0.4)$)  {};
\draw node[fill=red,circle,scale=0.45, label =right: $a_{3}$] at ($(A.corner 4)+(0, -0.8)$) {};
\draw node[fill=red,circle,scale=0.45, label = below:$a_{4}$] at ($(A.corner 3)+(1, 0)$) {};
\draw node[fill=white,scale=0, label = $1$] at ($(A)+(15:1.25)+(-0.35, 0.9)$) {};
\draw node[fill=white,scale=0, label = $2$] at ($(A)+(15:1.25)+(-2.1, 0.9)$) {};
\draw node[fill=white,scale=0, label = $3$] at ($(A)+(15:1.25)+(-2.1, -2.08)$) {};
\draw node[fill=white,scale=0, label = $4$] at ($(A)+(15:1.25)+(-0.35, -2.08)$) {};
\foreach \i in {1, 3}
    \node[circle, label=left:$v_{\i}$] at (A.corner \i) {};
\foreach \i in {2}
    \node[circle, label=below:$v_{\i}$] at (A.corner \i) {};
\foreach \i in {4}
    \node[circle, label=above:$v_{\i}$] at (A.corner \i) {};
\draw [->, dotted, thick] (A.center) -- ($(A.corner 1)-(A.corner 2)$) node[right]  {$v_{1}-v_{2}$};
\draw [->, dotted, thick] (A.center) -- ($0.5*(A.corner 1)-0.5*(A.corner 3)+(0, 1)$) node[right]  {$v_{1}-v_{3}$};
\draw [->, dotted, thick] (A.center) -- ($0.5*(A.corner 4)-0.5*(A.corner 2)+(1, 0)$) node[right]  {$v_{4}-v_{2}$};
\draw [->, dotted, thick] (A.center) -- ($(A.corner 4)-(A.corner 1)$) node[right]  {$v_{4}-v_{1}$};
\draw [->, dotted, thick] (A.center) -- ($0.5*(A.corner 3)-0.5*(A.corner 1)+(0, -1)$) node[right]  {$v_{3}-v_{1}$};
\draw [->, dotted, thick] (A.center) -- ($(A.corner 2)-(A.corner 1)$)  node[right]  {};
\draw [->, dotted, thick] (A.center) -- ($-0.5*(A.corner 4)+0.5*(A.corner 2)-(1, 0)$) node[right]  {};
\draw [->, dotted, thick] (A.center) -- ($(A.corner 1)-(A.corner 4)$) node[right]  {};
\end{tikzpicture}

    \label{fig:bis_fan_cross_2}
\end{subfigure}
~
\hspace{1 cm}
\begin{subfigure}[b]{0.45 \textwidth}
\centering
\includegraphics[scale=0.5]{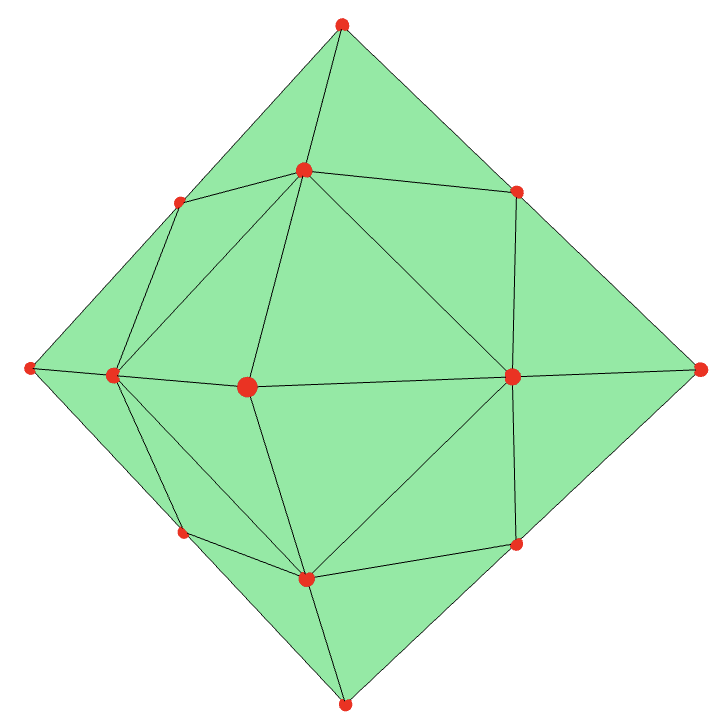}
    \label{fig:bisfan_octa}
\end{subfigure}
\caption{The bisection fans of the $\ell_{1}$-norm in dimensions $2$ and $3$. The points $a_{1}, a_{2}, a_{3}, a_{4}$ lie in different maximal cones of the bisection fan in $2$ dimensions. The effect on the bisectors $\bis(0, a_{i})$ is illustrated in Figure~\ref{fig:4_bisector}.} 
\label{fig:bis_fan_l1_norm}
\end{figure}

\begin{thm}\label{prop:bisfantree2}
The bisection fan of the cross-polytope is the polyhedral fan induced by the hyperplane arrangement \[
\mathcal{H} \ = \ \left\{\left\{x\in \mathbb{R}^d \, \colon \, x_{i} = 0 \right\}_{i =1}^{d} \bigcup \left\{x\in \mathbb{R}^d \, \colon \, x(I) = x(I^{c}) \right\}_{I \subseteq 2^{[d]}}\right\}
\] 
\end{thm}

\begin{proof}
To show that $\mathcal{H}$ is the bisection fan of the cross-polytope it suffices to show the following for points $a, b$ in general position in $\mathbb{R}^{d}$: \begin{enumerate}
    \item If $a, b$ lie in the interior of the same region of $\mathcal{H}$, then $\bis(0, a)$ is equivalent to $\bis(0, b).$
    
    \item If $a, b$ are separated by some hyperplane of $\mathcal{H}$, then $\bis(0, a)$ is not equivalent to $\bis(0, b).$
    
\end{enumerate}
For the first point, we observe that the inequalities used to describe the bisection cones in Proposition~\ref{bis_tree} correspond precisely to the hyperplanes in $\mathcal{H}.$ Hence, simultaneous membership of $a, b$ in a region of $\mathcal{H}$ implies simultaneous (non-)membership of $a, b$ in each bisection cone.

For the second point, this follows directly from Proposition~\ref{lem:bis_refine}.
\end{proof}

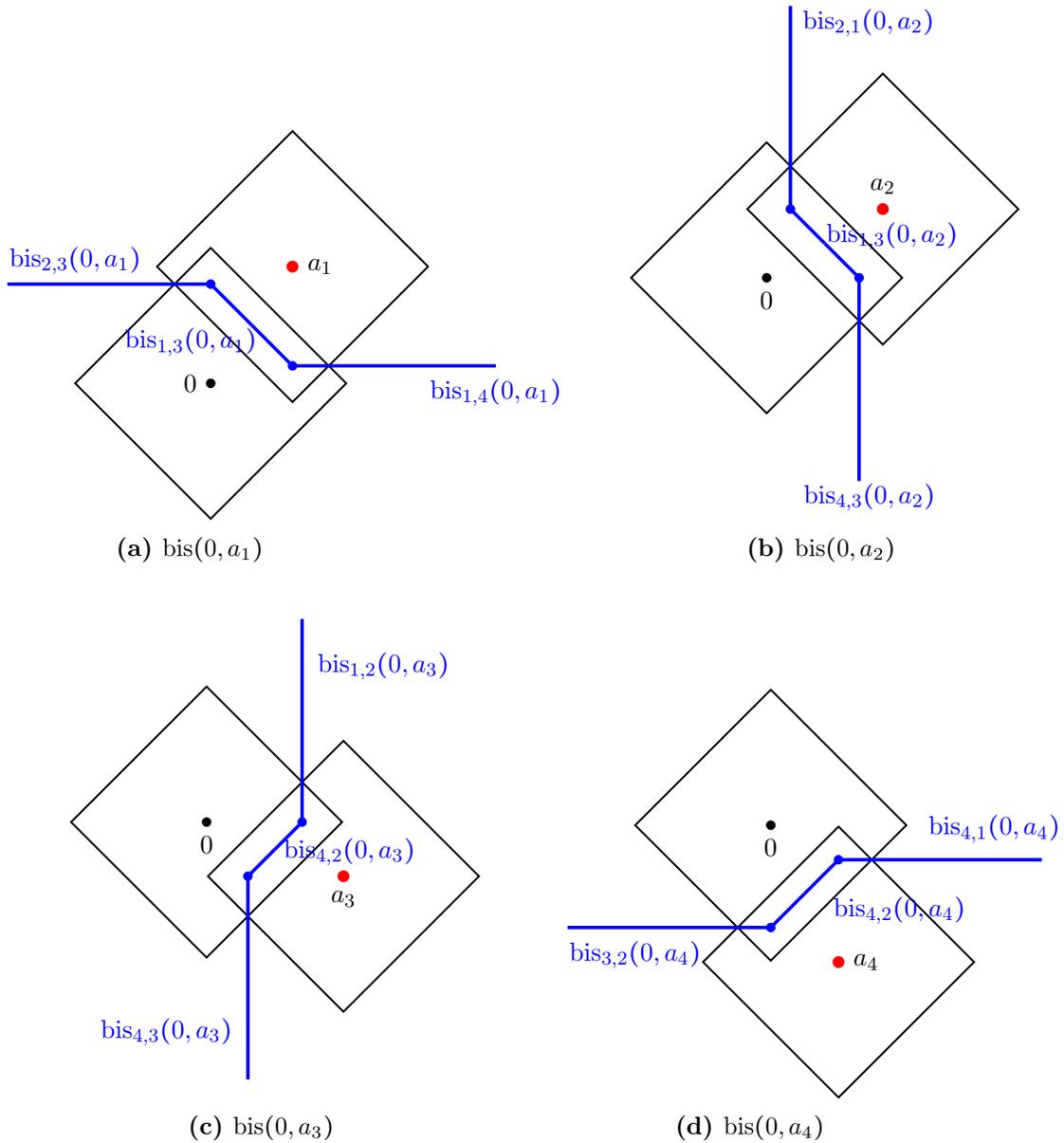
\begin{figure}[b]
        \centering
        \begin{subfigure}[b]{0.35\textwidth}
            \centering
            \scalebox{0.95}{
            \begin{tikzpicture}
\definecolor{vertexcolor_unnamed__1}{rgb}{ 1 0 0 }
 \definecolor{facetcolor_unnamed__1}{rgb}{ 0.4666666667 0.9254901961 0.6196078431 }
    \node [
    regular polygon, 
    regular polygon sides=4, 
    minimum size=4 cm, shape border rotate = 45, scale = 1,
    draw=black, thick
] 
at (0,0) (A){};
\node[dot=4pt, fill=black, label=left:{$0$}] at (A.center) {};
\node [
    regular polygon, 
    regular polygon sides=4, 
    minimum size=4 cm, shape border rotate = 45, scale = 1,
    draw=black, thick
] 
at ($(A)+(15:1.25)+(0, 1.4)$) (B){};
\node[dot=5pt, fill=red, label=right:{$a_{1}$}] at (B.center) {};
\path (A) -- (B) coordinate[midway] (M);
\path[name path=a_to_a_corner1] (A.center)--(A.corner 1); 
\path[name path=m_135] (M) -- ($(M)+(135:1)$);
\fill [blue, name intersections={of=a_to_a_corner1 and m_135}] (intersection-1) circle (2pt);
\draw[blue, thick, line width = 0.5mm] (M) -- (intersection-1);
\draw[blue, thick, line width = 0.5mm] (intersection-1) -- ($(intersection-1)+(-3, 0)$) node[below, label={[xshift=1cm, yshift=0cm]{$\bis_{2,3}(0, a_{1})$}}] {};

\path[name path=b_to_b_corner3] (B.center)--(B.corner 3); 
\path[name path=m_315] (M) -- ($(M)+(-45:2)$);
\fill [blue, name intersections={of=b_to_b_corner3 and m_315}] (intersection-1) circle (2pt);
\draw[blue, thick, line width = 0.5mm] (M) -- (intersection-1) node[below, label={[xshift=-1.5cm, yshift=0.0cm]{$\bis_{1,3}(0, a_{1})$}}] {};
\draw[blue, thick, line width = 0.5mm] (intersection-1) -- ($(intersection-1)+(3, 0)$) node[below, label={[xshift=0cm, yshift=-0.75cm]{$\bis_{1,4}(0, a_{1})$}}] {};
\end{tikzpicture}
}
            \caption[Network2]%
            {{\small $\bis(0, a_{1})$}}    
            \label{fig:mean and std of net14}
        \end{subfigure}
        \hspace{3 cm}
        \begin{subfigure}[b]{0.35\textwidth}  
            \centering 
            \scalebox{0.95}{
            \begin{tikzpicture}
\definecolor{vertexcolor_unnamed__1}{rgb}{ 1 0 0 }
 \definecolor{facetcolor_unnamed__1}{rgb}{ 0.4666666667 0.9254901961 0.6196078431 }
    \node [
    regular polygon, 
    regular polygon sides=4, 
    minimum size=4 cm, shape border rotate = 45, scale = 1,
    draw=black, thick
] 
at (0, 0) (A){};
\node[dot=4pt, fill=black, label=below:{$0$}] at (A.center) {};
\node [
    regular polygon, 
    regular polygon sides=4, 
    minimum size=4 cm, shape border rotate = 45, scale = 1,
    draw=black, thick
] 
at ($(A)+(45:2.0)+(0.3, -0.4)$) (B){};
\node[dot=5pt, fill=red, label=above:{$a_{2}$}] at (B.center) {};
\path (A) -- (B) coordinate[midway] (M);
\path[name path=a_to_a_corner1] (A.center)--(A.corner 4); 
\path[name path=m_135] (M) -- ($(M)+(-45:3)$);
\fill [blue, name intersections={of=a_to_a_corner1 and m_135}] (intersection-1) circle (2pt);
\draw[blue, thick, line width = 0.5mm] (M) -- (intersection-1);
\draw[blue, thick, line width = 0.5mm] (intersection-1) -- ($(intersection-1)+(0, -3)$) node[right, label={[xshift=0cm, yshift=-0.75cm]{$\bis_{4,3}(0, a_{2})$}}] {};

\path[name path=b_to_b_corner3] (B.center)--(B.corner 2); 
\path[name path=m_315] (M) -- ($(M)+(135:2)$);
\fill [blue, name intersections={of=b_to_b_corner3 and m_315}] (intersection-1) circle (2pt);
\draw[blue, thick, line width = 0.5mm] (M) -- (intersection-1) node[below, label={[xshift=1.5cm, yshift=-0.75cm]{$\bis_{1,3}(0, a_{2})$}}] {};
\draw[blue, thick, line width = 0.5mm] (intersection-1) -- ($(intersection-1)+(0, 3)$) node[right, label={[xshift=1cm, yshift=-0.75cm]{$\bis_{2,1}(0, a_{2})$}}] {};
\end{tikzpicture}
}
            
            \caption[]%
            {{\small $\bis(0, a_{2})$}}    
            \label{fig:mean and std of net24}
        \end{subfigure}
        \vskip\baselineskip
        \begin{subfigure}[b]{0.35\textwidth}   
            \centering 
            \scalebox{0.95}{
            \begin{tikzpicture}
\definecolor{vertexcolor_unnamed__1}{rgb}{ 1 0 0 }
 \definecolor{facetcolor_unnamed__1}{rgb}{ 0.4666666667 0.9254901961 0.6196078431 }
    \node [
    regular polygon, 
    regular polygon sides=4, 
    minimum size=4 cm, shape border rotate = 45, scale = 1,
    draw=black, thick
] 
at (0,0) (A){};
\node[dot=4pt, fill=black, label=below:{$0$}] at (A.center) {};
\node [
    regular polygon, 
    regular polygon sides=4, 
    minimum size=4 cm, shape border rotate = 45, scale = 1,
    draw=black, thick
] 
at ($(A.corner 4)+(0, -0.8)$) (B){};
\node[dot=5pt, fill=red, label=below:{$a_{3}$}] at (B.center) {};
\path (A) -- (B) coordinate[midway] (M);
\path[name path=a_to_a_corner1] (A.center)--(A.corner 4); 
\path[name path=m_135] (M) -- ($(M)+(45:3)$);
\fill [blue, name intersections={of=a_to_a_corner1 and m_135}] (intersection-1) circle (2pt);
\draw[blue, thick, line width = 0.5mm] (M) -- (intersection-1);
\draw[blue, thick, line width = 0.5mm] (intersection-1) -- ($(intersection-1)+(0, 3)$) node[above, label={[xshift=1.2cm, yshift=-1.35cm]{$\bis_{1,2}(0, a_{3})$}}] {};

\path[name path=b_to_b_corner3] (B.center)--(B.corner 2); 
\path[name path=m_315] (M) -- ($(M)+(225:2)$);
\fill [blue, name intersections={of=b_to_b_corner3 and m_315}] (intersection-1) circle (2pt);
\draw[blue, thick, line width = 0.5mm] (M) -- (intersection-1) node[above, label={[xshift=1.5cm, yshift=-0.3cm]{$\bis_{4,2}(0, a_{3})$}}] {};
\draw[blue, thick, line width = 0.5mm] (intersection-1) -- ($(intersection-1)+(0, -3)$) node[below, label={[xshift=-1.2cm, yshift=0.35cm]{$\bis_{4,3}(0, a_{3})$}}] {};
\end{tikzpicture}
}
           
            \caption[]%
            {{\small $\bis(0, a_{3})$}}    
            \label{fig:mean and std of net34}
        \end{subfigure}
        \hspace{1 cm}
            \vspace{1 cm}
        \begin{subfigure}[b]{0.35\textwidth}   
            \centering 
            \scalebox{0.95}{
\begin{tikzpicture}
\definecolor{vertexcolor_unnamed__1}{rgb}{ 1 0 0 }
 \definecolor{facetcolor_unnamed__1}{rgb}{ 0.4666666667 0.9254901961 0.6196078431 }
    \node [
    regular polygon, 
    regular polygon sides=4, 
    minimum size=4 cm, shape border rotate = 45, scale = 1,
    draw=black, thick
] 
at (0,0) (A){};
\node[dot=4pt, fill=black, label=below:{$0$}] at (A.center) {};
\node [
    regular polygon, 
    regular polygon sides=4, 
    minimum size=4 cm, shape border rotate = 45, scale = 1,
    draw=black, thick
] 
at ($(A.corner 3)+(1, 0)$) (B){};
\node[dot=5pt, fill=red, label=right:{$a_{4}$}] at (B.center) {};
\path (A) -- (B) coordinate[midway] (M);
\path[name path=a_to_a_corner1] (A.center)--(A.corner 3); 
\path[name path=m_135] (M) -- ($(M)+(225:3)$);
\fill [blue, name intersections={of=a_to_a_corner1 and m_135}] (intersection-1) circle (2pt);
\draw[blue, thick, line width = 0.5mm] (M) -- (intersection-1);
\draw[blue, thick, line width = 0.5mm] (intersection-1) -- ($(intersection-1)+(-3, 0)$) node[below, label={[xshift=1cm, yshift=-0.75cm]{$\bis_{3,2}(0, a_{4})$}}] {};

\path[name path=b_to_b_corner3] (B.center)--(B.corner 1); 
\path[name path=m_315] (M) -- ($(M)+(45:2)$);
\fill [blue, name intersections={of=b_to_b_corner3 and m_315}] (intersection-1) circle (2pt);
\draw[blue, thick, line width = 0.5mm] (M) -- (intersection-1) node[above, label={[xshift=0.9cm, yshift=-1.4cm]{$\bis_{4,2}(0, a_{4})$}}] {};
\draw[blue, thick, line width = 0.5mm] (intersection-1) -- ($(intersection-1)+(3, 0)$) node[above, label={[xshift=-0.7cm, yshift=-0.2cm]{$\bis_{4,1}(0, a_{4})$}}] {};
\end{tikzpicture}
}
            
            \caption[]%
            {{\small $\bis(0, a_{4})$}}    
            \label{fig:mean and std of net44}
        \end{subfigure}
        \caption[.]
        {\small The bisector $\bis(0, a_{i})$ for $i=1, \ldots ,4$ where $a_{1},a_{2},a_{3},a_{4}$ are $4$ points lying in different maximal cones of the bisection fan; see Figure~\ref{fig:bis_fan_l1_norm}. No pair of bisectors has the same labelled maximal cells. This shows that the bisectors are pairwise non-equivalent.}
        \label{fig:4_bisector}
    \end{figure} 

\section{Discrete Wasserstein norm}\label{sec:Wasserstein}

In this section we study bisectors of the discrete Wasserstein distance $\dist^{P}(0, a)$ defined on the hyperplane $
H:=\left\{x\in \mathbb{R}^{d} \, \colon \, x([d]) = 0\right\}$ by \[
\max \, \, \inner{a}{x} \quad \text{subject to} \quad |x_{i}-x_{j}| \leq 1 \, \, \text{for $1 \leq i < j \leq d.$}
\] 

As in the previous section, we use the notation $
x(I) \coloneqq \sum_{i 
\in I}x_{i}$. The corresponding unit ball is the polytope 
\[
P \ = \ \conv \{\pm (e_i-e_j) \colon 1 \leq i < j \leq d\} \, ,
\]

 the root polytope of the lattice $A_d$  which can equivalently be seen as the symmetric edge polytope of the complete graph on $d$ vertices, see, e.g., \cite{ardila2011root,symedgepol}.

The polytope has dimension $d-1.$ Its face fan is induced by the intersection of $H$ with the coordinate hyperplanes in $\mathbb{R}^d$. In particular, the facets of $P$ are in one-to-one correspondence with proper subsets $\emptyset \neq I \subsetneq [d]$ (See, e.g., \cite[Corollary 3.3]{symedgepol}). More precisely, for every proper subset $I$, the set $F_I:=\conv \{e_i -e_j \colon i\in I, j\not \in I\}$ is a facet of $P$. The cone over the facet $F_I$ is equal to
\[
C_I \ := \ C_{F_I} \ = \ H\cap \mathcal{O}_{I} \, ,
\]
where $\mathcal{O}_{I}=\{x \in \mathbb{R}^{d} \, : \, x_{i} \geq 0 \text{ for all } i\in I \, \, \text{and} \, \, x_{i} \leq 0 \, \, \text{for all} \,\, i \in I^{c}\}$. For any $x=\sum \lambda _{ij} (e_i-e_j) \in C_I$, $\lambda _{i,j}\geq 0$, we have by linearity
 \[
 \dist \nolimits^{P}(0,x) \ = \ \sum_{(i, j) \in I \times I^{c}} \lambda_{i, j} \ = \ \sum_{i\in I}x_i \ = x(I) =\ \frac{\|x\|_1}{2} \, .
 \]
See also~\cite[Corollary 8]{ccelik2021wasserstein}. Note that the last equality above follows from the fact that $x(I) = -x(I^{c})$; we will use this repeatedly hereafter.

The inequality description of the bisection cones, denoted $\mathcal{B}_{I, J}=\mathcal{B}_{F_I, F_J}$ for brevity, is given by the following proposition.

\newpage

\begin{prop}\label{biscone_kd}
For proper subsets $\emptyset \neq I, J \subsetneq [d]$, the bisection cone $\mathcal{B}_{I, J}$ consists of all $a \in H$ satisfying the following inequalities. 

    \begin{minipage}{0.45\textwidth}
    \begin{itemize}
        \item[(I)] If $I \cap J \neq \emptyset$ and $I^c \cap J^c \neq \emptyset$ then 
        \begin{enumerate}[(i)]
            \item $a_{i} \geq 0 \quad \forall \, i \in I \cap J^c$

            \item $a_{i} \leq 0 \quad \forall \, i \in J \cap I^c$

            \item $ a(I) \geq 0$

            \item $a(J) \leq 0$
        \end{enumerate}
        \item[(II)]If $I \cap J = \emptyset$ and $I^c \cap J^c \neq \emptyset$ then
        \begin{enumerate}[(i)]
            \item $a_{i} \geq 0 \quad \forall \, i \in I$
            \item $a_{i} \leq 0 \quad \forall \, i \in J$
            \item $a(J) \leq a(K) \quad \forall \, K \subseteq I^c \cap J^c$
            \item $a(K) \leq a(I) \quad \forall \, K \subseteq I^c \cap J^c$
        \end{enumerate}
        \end{itemize}
        \end{minipage}
        \begin{minipage}{0.45\textwidth}
        \vspace{0.38cm}
        \begin{itemize}
    \item[(III)] If $I \cap J \neq \emptyset$ and $I^c \cap J^c = \emptyset$ then
        \begin{enumerate}[(i)]
            \item $a_{i} \geq 0 \quad \forall \, i \in J^c$
            \item $a_{i} \leq 0 \quad \forall \, i \in I^c$
            \item $a(K) \leq a(J^{c}) \quad \forall \, K \subseteq I \cap J$
            \item $a(I^{c}) \leq a(K) \quad \forall \, K \subseteq I \cap J$
        \end{enumerate}
    \item[(IV)]If $I \cap J = \emptyset$ and $I^c \cap J^c = \emptyset$ then
        \begin{enumerate}[(i)]
        \item $a_{i} \geq 0 \quad \forall \, i \in I $
        \item $a_{i} \leq 0 \quad \forall \, i \in J $
        
        \end{enumerate}
        \end{itemize}
        \vspace{1.25cm}
        \end{minipage}

\end{prop}

\begin{proof}
\textsc{Necessity:}\\
We begin by proving that any $a\in \mathcal{B}_{I, J}$ necessarily satisfies the conditions above. By definition, if $a\in \mathcal{B}_{I, J}$ then $\bis (0,a)\cap \mathcal{O}_I\cap (\mathcal{O}_J+a)$ is non-empty. That is, there exists an $x\in \mathbb{R}^d$ such that
\begin{itemize}
\item[(a)]$x([d])=0$,
\item[(b)] $x(I)=x(J) -a(J)$,
\item[(c)] $x_i\geq 0$ for all $i\in I$ and $x_i\leq 0$ for all $i\in I^c$,
\item[(d)] $x_i\geq a_i$ for all $i\in J$ and $x_i\leq a_i$ for all $i\in J^c$.
\end{itemize}

From (c) and (d) it follows that 
\begin{eqnarray}
    a_{i} \geq  x_i\geq 0 && \text{ for all } i \in I \cap J^{c} \text{ , and } \label{eq:ineq1}\\ 
    a_{i} \leq  x_i\leq 0 &&\text{ for all } i \in J \cap I^{c} \, . \label{eq:ineq2}
\end{eqnarray}

In particular, these inequalities imply $a_{i} \geq 0$ for all $i \in I \cap J^{c}$ and $a_{i} \leq 0$ for all $i \in J \cap I^{c}$. These inequalities are exactly recorded in condition (i) and (ii) in all four cases, adjusted to account for whether $I\cap J$ and $I^c\cap J^c$ are empty or not.

To prove the necessity of the remaining conditions we proceed by a case-by-case distinction.

Case (I): Summing the inequalities~\eqref{eq:ineq1} and \eqref{eq:ineq2} over the elements in $I\cap J^c$ and $I^c\cap J$ we obtain
\[0 \leq x(I \cap J^{c}) \leq a(I \cap J^{c}) \quad \text{and} \quad 0 \leq  - x(J \cap I^{c}) \leq - a(J \cap I^{c}) \, .
\]
After summing these two inequalities and using (b), we get
\[
0 \leq  x(I \cap J^{c}) - x(J \cap I^{c})  = - a(J) \leq a(I \cap J^{c}) - a(J \cap I^{c}) = a(I) - a(J) \, .
\]
From here, it follows that $a(I) \geq 0$ and $a(J) \leq 0$, which are conditions (iii) and (iv) respectively.

Case (II): From (c) and (d) it follows that $x_i\leq \min \{0,a_i\}$ for all $i\in I^c\cap J^c$. Thus, for all $K\subseteq I^c\cap J^c$, we have
\begin{equation}\label{kd2}
x(I^{c} \cap J^{c}) \leq a(K) \, .
\end{equation}
From $I\cap J=\emptyset$ and (a) we obtain
\begin{align}
 x(I)+x(J)&=  -x(I^{c} \cap J^{c}) \, . \label{biseqn}
        \end{align}
        Subtracting (b) from equation~\eqref{biseqn} we get
        \begin{align}
        2\cdot x(J)&= -x(I^{c} \cap J^{c})+a(J) \label{biseqn2} \, .
        \end{align}
        Since $J\subseteq I^c$, from (c) it follows that  $x(J) \leq 0$. This together with equations~\eqref{biseqn2} and~\eqref{kd2} yields condition (iii):
        \[
a(J) \leq x(I^{c} \cap J^{c}) \leq a(K) \, .
        \]
        On the other hand, adding (b) to equation~\eqref{biseqn} we get
        \begin{align}
        2\cdot x(I)&= -x(I^{c} \cap J^{c} ) - a(J) \label{biseqn3} \, .
        \end{align}
        Since $I\subseteq J^c$, from (c) it follows that  $x(I) \leq a(I)$. This together with equation~\eqref{biseqn3} and~\eqref{kd2} applied to $K'=I^c\cap J^c\cap K^c$ yields
        \[
        2\cdot a(I)\geq 2\cdot x(I)= -x(I^{c} \cap J^{c})- a(J)\geq - a(I^c \cap J^c\cap K^{c})- a(J)
        \]
        which after rearrangement gives
                \[
        a(I)\geq -a(I^c \cap J^c\cap K^c)- a(J)-a(I)=a(K),
        \]
        since $a\in H$. This shows that condition (iv) is satisfied.

         Cases (III) and (IV): We observe that $F_I=-F_{I^c}$ and thus  $\mathcal{B}_{I, J}=-\mathcal{B}_{I^c,J^c}$ for all proper subsets $I$ and $J$ of $[d]$. Thus, Case (III) follows from Case (II) by replacing $I$ with $I^c$ and $J$ with $J^c$. Case (IV) follows from equations~(\ref{eq:ineq1}) and~(\ref{eq:ineq2}) above.

\textsc{Sufficiency:}\\
In order to prove sufficiency of the conditions we construct an $x\in \bis (0,a)\cap \mathcal{O}_I\cap (\mathcal{O}_J+a)$ for any $a\in H$ that satisfies the conditions above. Again, we proceed by a case-by-case distinction.

Case (I): For any $a\in H$ that satisfies the conditions (i)-(iv) we consider the quantity
\[ 
D = a(I \cap J^{c}) - a(J \cap I^{c}) \, .
\]
Then $D \geq 0$, by conditions (i) and (ii). Further, we define 
\[
\lambda = \begin{cases}\frac{1}{D}\left(-a(J)\right) & \text{ if }D>0 \, , \\
0 & \text{ if } D=0 \, .
\end{cases}
\]
Then we consider $x \in \mathbb{R}^{d}$ defined by  \[
    x_{i} =  \begin{cases} 
\lambda a_{i} & i \in (J \cap I^{c}) \cup (I \cap J^{c}), \\ b_{i} & i \in I \cap J, \\ c_{i} & i \in I^{c} \cap J^{c} \, ,
\end{cases} 
\] 

 where $b_{i} \geq \max\{0, a_{i}\}$ for $i \in I \cap J$ and $c_{i} \leq \min\{0, a_{i}\}$ for $i \in I^{c} \cap J^{c}$ are chosen such that \begin{equation}\label{maxmin}
 \sum_{i \in I \cap J}b_{i} + \sum_{i \in I^{c} \cap J^{c}}c_{i} = -\lambda \left(\sum_{i \in I \cap J^{c}}a_{i} + \sum_{i \in J \cap I^{c}}a_{i} \right),
 \end{equation}
 
 in order for condition (a) to be satisfied. Such a choice of $b_i$s and $c_i$s is always possible by the non-emptiness of $I \cap J$ and $I^{c} \cap J^{c}$, and the opposite sign of $b_{i}$ and $c_{i}.$ By definition, any such $x$ lies in $H$, that is, satisfies condition (a). Since $a$ satisfies conditions (iii) and (iv) it follows that $\lambda \in [0,1]$. From that and since $a$ satisfies conditions (i) and (ii), $x$ lies in $x \in \mathcal{O}_{I} \cap (\mathcal{O}_{J}+a)$. That is, $x$ satisfies conditions (c) and (d), in both cases $D=0$ and $D>0$. To see that condition (b) is also satisfied by $x$ we distinguish between the cases $D>0$ and $D=0$.

If $D>0$ then by definition of $x$, $D$ and $\lambda$ we have
\begin{equation}\label{eq:bequiv}
x(I \cap J^{c}) - x(J \cap I^{c}) = \lambda (a(I \cap J^{c}) - a(J \cap I^{c})) = \lambda D = - a(J)
\end{equation}
which is precisely condition (b). 

In the other case, if $D=0$, then since $a$ satisfies conditions (i) and (ii), it follows that $a_{i}=0$ for all $i \in (I \cap J^{c}) \cup (J \cap I^{c}).$ Since $a$ satisfies conditions (iii) and (iv) we therefore also have the following chain of inequalities: 
\[
0 \geq a(J) = a(J \cap I) = a(I) \geq 0 \, .
\]
In particular, $a(J) = 0$ and therefore both sides of equation~\eqref{eq:bequiv} are equal to $0$. That is, condition (b) is satisfied by $x$. In summary, $x$ satisfies conditions (a)-(d) and therefore lies in $\bis (0,a)\cap \mathcal{O}_I\cap (\mathcal{O}_J+a)$ as required.

Case (II): In order to prove sufficiency in this case we distinguish between three further cases: 

($\alpha$) $a(I)=0$: In this case, by condition (i), we have $a_i=0$ for all $i\in I$. Further, by conditions (ii) and (iv) (applied to $K = I^{c} \cap J^{c}$) we have the chain of inequalities
\[
 0 \geq a(J) = - a(J^{c}) = -a(J^{c} \cap I) -a(J^{c} \cap I^{c}) \geq -a(I)=0 \, .
 \]
 In particular, we must also have $a(J)=0$, and thus, together with condition (ii) we must have $a_i=0$ for all $i\in J$. Now, setting $K=\{i\}$, $i\in J^{c} \cap I^c$, it follows from conditions (iii) and (iv) that we also have $a_i=0$ for all $i\in I^c\cap J^c$. In summary, if $a(I)=0$ then $a=0$. In this case, clearly $x=0$ satisfies the conditions (a)-(d).

 ($\beta$) $a(I) > 0$ and $-a(J) \leq a(I)$ : In this case we consider
 \[
\lambda = \frac{-a(J)}{a(I)}
 \]
 which, by assumption and conditions (i) and (ii), lies in the interval $[0,1]$. We then define
 \[
    x_{i} =  \begin{cases} 
0 & i \in J, \\ \lambda a_{i} & i \in I,  \\ c_{i} & i \in I^{c} \cap J^{c},
            \end{cases}
            \]
            
where $c_{i} \leq \min\{0, a_{i}\}$ for each $i \in I^{c} \cap J^{c}$ are chosen such that $\sum_{i \in I^{c} \cap J^{c}}c_{i} = a(J)$. This is always possible. To see that, we consider $\tilde{c}_i=\min\{0, a_{i}\}$. Then, by condition (iii), $\sum _{i\in I^c\cap J^c}\tilde{c}_i\geq a(J)$. If the inequality is strict, then we may choose an $i_0\in I^c\cap J^c$ (which is non-empty) and set \[c_{i}=\begin{cases} \tilde{c}_i & i\in (I^c\cap J^c)\setminus \{i_0\}, \\
a(J)-\sum \limits_{i\in (I^c\cap J^c) \setminus \{i_{0}\}}\tilde{c}_i &i = i_{0}.
\end{cases}
\]
By construction, $x$ satisfies condition (a). Further, since $0\leq \lambda \leq 1$ we see that also conditions (c) and (d) are satisfied by $x$. And it also satisfies condition (b) since 
\[x(I) - x(J) = \lambda a(I) = -a(J) \, .
\]

($\gamma$) $a(I) > 0$ and $-a(J) > a(I)$: Then we have $0 < a(I)<-a(J)$, in particular, we may define
\[
\mu = \frac{a(I)+a(J)}{a(J)},
\]
which by assumption lies in the interval $(0,1]$. We then consider
\[x_{i} =  \begin{cases} 
a_{i} & i \in I, \\ \mu a_{i} & i \in J,  \\ c_{i} & i \in I^{c} \cap J^{c}.
            \end{cases} \, 
            \]
where $c_{i} \leq \min\{0, a_{i}\}$ for each $i \in I^{c} \cap J^{c}$ are chosen such that $\sum_{i \in I^{c} \cap J^{c}}c_{i} = -a(I) +a(I^{c} \cap J^{c})$. Condition (iv) ensures that such a choice is always possible: as in case ($\beta$) we can initially set $
        \tilde{c_{i}} = \min\{0, a_{i}\}$ for $i \in I^{c}\cap J^{c}$. By condition (iv), for any subset $K\subseteq I^{c}\cap J^{c}$, $-a(I) \leq - \sum_{i \in (I^{c} \cap J^{c})\setminus K}a_{i}$ which is equivalent to $-a(I) +a(I^{c} \cap J^{c}) \leq a(K)$. In particular,
        \[
        \sum _{i\in I^{c} \cap J^{c}}\tilde{c}_i\geq -a(I) +a(I^{c} \cap J^{c}) \, . 
        \]
        We then define $c_i$ similarly as in ($\beta$) by
        \[c_{i}=\begin{cases} \tilde{c}_i & i\in (I^c\cap J^c)\setminus \{i_0\}, \\
-a(I) +a(I^{c} \cap J^{c})-\sum \limits_{i\in (I^c\cap J^c) \setminus \{i_{0}\}}\tilde{c}_i &i = i_{0}.
\end{cases}
\]
By construction, $x$ then satisfies condition (a), and since $0 < \mu \leq 1$ we observe that $x$ also satisfies conditions (c) and (d). Also condition (b) is satisfied by $x$ since
\[
x(I) - x(J) = a(I) - \mu a(J) = a(I) - \left(a(I) + a(J)\right) = -a(J) \, .
\]
In summary, $x$ satisfies conditions (a)-(d), as desired.

Case (III): Follows from Case (II) as argued above in the proof of the necessity of the conditions.

Case (IV): Similarly as in Case (I) we consider the quantity $D=a(I) - a(J)$. If $D=0$ then, by conditions (i) and (ii), we have $a_i=0$ for all $i\in I\cup J=[d]$. In this case, $x=0$ clearly satisfies all conditions (a)-(d). If $D > 0$ we define $\lambda =1/D \cdot (-a(J))$ and set $x=\lambda a$. Then $x$ satisfies condition (a). By conditions (i) and (ii) and the sign of $D$, $\lambda \in [0,1]$ and thus we see that conditions (c) and (d) are also satisfied. Lastly, condition (b) is also satisfied since 
\[
x(I) - x(J) = \lambda \left(a(I) - a(J)\right) = -a(J)\, .
\]
Again, we have constructed an $x$ with the desired properties. This finishes the proof.
\end{proof}

We will now use Proposition~\ref{biscone_kd} to give a combinatorial and a geometric description of the equivalence classes of bisectors with respect to the Wasserstein distance thereby showing that the bisection fan exists. To that end, we begin by associating to every point $a\in H=\{x\in \mathbb{R}^d: x([d]) = 0\}$ the following sets and set systems.
\begin{align*}
I_{a}^{+} &= \{i \in [d]: a_{i}>0\} &&\textbf{positive support}\\
I_{a}^{-} &= \{i \in [d]: a_{i}<0\} &&\textbf{negative support}\\
S_{a}^{+} &= \{K \subset [d]: a(K) > 0\} &&\textbf{positive sums}\\
S_{a}^{-} &= \{K \subset [d]: a(K) < 0\} &&\textbf{negative sums}
 \end{align*}

  We may identify $I_a^+$ and $I_a^-$ with the one-element subsets in $S_a^+$ and $S_a^-$, respectively. Further, we consider the system of \textbf{light/heavy positive/negative sums} as defined in the following table.

 \begin{center}
\begin{tabular}{ |p{1.5cm}|p{6cm}|p{6cm}|}
\hline
& \textbf{Positive} & \textbf{Negative}  \\
\hline
\textbf{Light} & $I_{a,<}^{+}=\{X\subseteq I_{a}^{+}\colon a(X)<a(I_{a}^{+}\setminus X)\}$ & $I_{a,<}^{-}=\{X\subseteq I_{a}^{-}\colon a(X)<a(I_{a}^{-}\setminus X)\}$\\
\hline
\textbf{Heavy} & $I_{a,>}^{+}=\{X\subseteq I_{a}^{+}\colon a(X)>a(I_{a}^{+}\setminus X)\}$ &  $I_{a,>}^{-}=\{X\subseteq I_{a}^{-}\colon a(X)>a(I_{a}^{-}\setminus X)\}$ \\
\hline
\end{tabular}
\end{center}

We call a point $a\in H$ \textbf{generic} if 
 \begin{align*}
 S_a^+\sqcup S_a^-&=2^{[d]}\setminus \{\emptyset, [d]\} \, , \\
 I_{a,>}^+\sqcup I_{a,<}^+&=2^{I_a^+}\, , \text{ and} \\
 I_{a,>}^-\sqcup I_{a,<}^-&=2^{I_a^-} \, .
 \end{align*}
In other words, a point is generic if every proper sub-sum of its coordinates is non-zero, and any sub-sum of the positive (negative) support is never equal in absolute value to that of its complement in the positive (negative) support. It follows that for generic $a\in H$, the set systems defined above are determined only by the data 
\[
D(a) := (S_a^+, I_{a,>}^+,I_{a,<}^-)
\]
consisting of the positive sums, heavy positive sums and light negative sums. With these definitions, we have the following combinatorial characterization of equivalence of bisectors.
\begin{thm}\label{set_sys}
For generic points $a,b\in H$, the bisectors $\bis^{P} (0,a)$ and $\bis^{P}(0,b)$ with respect to the discrete Wasserstein distance are equivalent if and only if $D(a)=D(b)$.
\end{thm}
\begin{proof}
In the following proof, every reference of case \Romannum{1}, \Romannum{2}, \Romannum{3}, \Romannum{4} is made with respect to the four cases in Proposition \ref{biscone_kd}. 

We begin by showing that if $\bis^{P}(0,a)$ and $\bis^{P}(0,b)$ are equivalent then $D(a)=D(b)$ for any generic $a,b\in H$. That is, we need to show that in this case $S_a^+=S_b^+$, $I_{a,>}^+=I_{b,>}^+$ and $I_{a,<}^-=I_{b,<}^-$.

$S_a^+=S_b^+$: First, we observe that for $I=I_a^{+}$ and $J=I_a^{-}$, by Proposition~\ref{biscone_kd} case (IV), we have that $a\in \mathcal{B}_{I,J}$. Since, by assumption, $\bis^{P}(0,a)$ and $\bis^{P}(0,b)$ are equivalent we therefore must also have $b\in \mathcal{B}_{I,J}$ which in turn implies $b_i>0$ for all $I_a^+$, that is, $I_a^{+} \subseteq I_b^{+}$. By symmetry of the argument we also obtain $I_b^{+}\subseteq I_a^{+}$ and thus equality of $I_a^{+}$ and $I_b^{+}$. Next, we observe that any non-empty subset of $I_a^+=I_b^+$ is contained in both $S_a^+$ and $S_b^+$. In order to show that $S_a^+$ equals $S_b^+$ it therefore suffices to prove that any subset $L\in S_a^+$ with $L\cap I_a^-\neq \emptyset$ is also contained in $S_b^+$. To see that we set $I=L$ and $J=I_a^-$. Then $I\cap J\neq \emptyset$. If $I\cup J=[d]$ then $I_a^+\subset L$ and therefore $L^c\subset I_a^-=I_b^-$. In particular, $L^c\in S_b^-$ and, equivalently, $L\in S_b^+$. Otherwise, if $I\cup J\neq [d]$ then $I^c\cap J^c\neq \emptyset$. Then $a$ satisfies the four conditions in Proposition~\ref{biscone_kd} case (I): $a(I)>0$ since $I=L\in S_a^+$ by assumption, and also conditions (i),(ii) and (iv) are satisfied for $J=I_a^-$. Thus, $a\in \mathcal{B}_{I,J}$. Since $\bis^{P}(0,a)$ and $\bis^{P}(0,b)$ are equivalent we therefore also have $b\in \mathcal{B}_{I,J}$ which in turn implies $L\in S_b^+$ by Proposition~\ref{biscone_kd} case (I) condition (iii).

$I_{a,>}^+=I_{b,>}^+$: Let $L\in I_{a,>}^+$. Then $L\subseteq I_a^+=I_b^+$ and $a(L)> a(I_a^{+}\setminus L)$. If $L=I_a^+=I_b^+$ then clearly $L$ is in $I_{a,>}^+$ as well as $I_{b,>}^+$. We may thus assume that $L\subset I_a^+$ is a proper subset. Again, we set $I=L$ and $J=I_a^-$. Then $I\cap J=\emptyset$ but $I^c\cap J^c\neq \emptyset$ since $L$ is a proper subset of $I_a^+$. Then $a$ satisfies all conditions in Proposition~\ref{biscone_kd} case (II): conditions (i)-(iii) are satisfied since $J=I_a^-$ and $I^c\cap J^c\subset I_a^+$, and condition (iv) is satisfied since $I^c\cap J^c=I_a^+\setminus L$. Thus, $a\in \mathcal{B}_{I,J}$ and therefore we also must have $b\in \mathcal{B}_{I,J}$ which in turn by Proposition~\ref{biscone_kd} case (II) condition (iv) applied to $K=I_a^+\setminus L$ implies that $L\in I_{b,>}^+$. 

$I_{a,<}^-=I_{b,<}^-$: Follows from the previous argument by considering $-a$ and $-b$.

For the other direction, we need to show that whenever $S_a^+=S_b^+$, $I_{a,>}^+=I_{b,>}^+$ and $I_{a,<}^-=I_{b,<}^-$, we have that $a\in \mathcal{B}_{I,J}$ if and only if $b\in \mathcal{B}_{I,J}$ for all proper subsets $I,J\subset [d]$.

We observe that conditions (i) and (ii) in all four cases in Proposition~\ref{biscone_kd} as well as conditions (iii) and (iv) in case (I) are simultaneously satisfied or violated by $a$ and $b$, since $S_a^+$ equals $S_b^+$. In particular, it follows that $a\in \mathcal{B}_{I,J}$ if and only if $b\in \mathcal{B}_{I,J}$ for all $I,J$ satisfying the conditions (I) and (IV) in Proposition~\ref{biscone_kd}. To prove the equivalence of $\bis^{P}(0,a)$ and $\bis^{P} (0,b)$ it thus only remains to show that conditions (iii) and (iv) in cases (II) and (III) are simultaneously satisfied/violated by $a$ and $b$. 

Let $a\in \mathcal{B}_{I,J}$ for proper subsets $I$ and $J$ that satisfy case (II), that is, $I\cap J= \emptyset$ and $I^c\cap J^c\neq \emptyset$. By condition (i), $I_a^-\subseteq I^c$. Thus, by setting $K=I^c\cap J^c\cap I_a^-$ in condition (iii) of case (II) we see that $J\in I_{a,<}^-$. By assumption, we also have $J\in I_{b,<}^-$ which implies that $b$ satisfies condition (iii) of case (II) for all subsets $K\subseteq I^c\cap J^c$. Similarly, we see that $b$ must also satisfy condition (iv) if it is satisfied by $a$ since $I_{a,>}^+=I_{b,>}^+$. This also shows that $b\in \mathcal{B}_{I,J}$. The argument for the equivalence of $a\in \mathcal{B}_{I,J}$ and $b\in\mathcal{B}_{I,J}$ for subsets $I,J$ that satisfy the assumptions of case (III) in Proposition~\ref{biscone_kd} is analogous.
\end{proof}

Theorem~\ref{set_sys} now allows us to give a polyhedral description of the equivalence classes of points in general position in $H$.

\begin{cor}\label{cor:dissectionWasserstein}
    For any point $a\in H$, $a$ is in general position if and only if $a$ is generic. The closure of the sets
    \[
    F_a \ = \ \{x\in H \, \colon \, x \text{ generic and }D(x)=D(a)\}
    \]
    form a dissection of $H$ into polyhedral cones.
\end{cor}
\begin{proof}
By definition, the set of generic points is an open and dense subset of $H$. By Theorem~\ref{set_sys} this set is partitioned according to the value of $D(a)$, that is, into sets of the form $F_a$. From their definition it follows that these sets are open polyhedral cones. Since their union is dense, their closures form a dissection of $H$. 

By Theorem~\ref{set_sys}, any point in $F_a$ is in general position since $F_a$ is open and the equivalence class of the bisector is constant on $F_a$. This shows that any generic point is in general position. On the other hand, any non-generic point lies on the boundary of some $F_a$ since those form a dissection of $H$. In particular, the equivalence class of such a point is not stable under small perturbations and thus the point is not in general position.
\end{proof}

By Corollary~\ref{cor:dissectionWasserstein} the equivalence classes of points in general position are given by the interiors of polyhedral cones that form a dissection of $H$. The next result shows that this dissection is, in fact, a polyhedral subdivision. We give a description as the common refinement of the \textbf{resonance arrangement} $\mathcal{H}$ (restricted to $H$) which consists of all hyperplanes
\[
x(I) \ = \ 0 \, , \quad \emptyset \neq I \subseteq [d]\, ,
\]
 and the regions of linearity of a piecewise linear function that we explicitly describe. In particular, the Wasserstein distance admits a bisection fan.

\begin{thm}~\label{thm:WassersteinPolyhedral}
Let $\mathcal{G}$ be the polyhedral fan induced by the regions of linearity of the function
\[
p(a) = \sum _{I\subseteq [d]} f_I (a)+g_I(a)
\]
where
\[
f_I (a) \ = \ \max \{\{a(K)\}_{K\subseteq I}, \{a(K)\}_{K\subseteq I^c}\}
\]
and
\[
g_I (a) \ = \ \min \{\{a(K)\}_{K\subseteq I}, \{a(K)\}_{K\subseteq I^c}\} \, .
\]
Then the bisection fan $\Delta$ of the discrete Wasserstein distance is given by the common refinement of $\mathcal{G}$ and $\mathcal{H}$.
\end{thm}

\begin{proof}
Let $a, b$ be in general position. By Theorem \ref{set_sys} and Corollary~\ref{cor:dissectionWasserstein}, we need to show that $D(a) = D(b)$ if and only if $a, b$ lie in the interior of the same maximal cone of $\Delta.$ For the only if direction, consider $F_{a}.$ By definition, for all points $x$ in $F_a$ and all proper subsets $I$ we have $x(I)>0$ if and only if $I\in S_a^+$. In particular, $F_a$ is contained in a region of $\mathcal{H}$. Furthermore, we observe that the set that maximises $a(K)$ in the definition of $f_I$ is either $I\cap I_a^+$ or at $I^c\cap I_a^+$, depending on if $I\cap I_a^+\in I_{a,>}^+$ or $I^c\cap I_a^+\in I_{a,>}^+$. Similarly, the set that minimizes $a(K)$ in the definition of $g_I$ is either $I\cap I_a^-$ or $I^c\cap I_a^-$, depending on if $I\cap I_a^-\in I_{a,<}^-$ or $I^c\cap I_a^+\in I_{a,<}^-$. That is, both $f_I$ and $g_I$ are linear on $F_a$ for all $I\subseteq [d]$ and so is their sum. This means that $F_{a}$ must be contained in some region of linearity of the function $p$ which is equal to some maximal cone of $\mathcal{G}.$  In particular, the common refinement of $\mathcal{H}$ and $\mathcal{G}$ coarsens the equivalence classes of bisectors.

    It remains to show that for points $a,b$ in general position and $D(a)\neq D(b)$ we have that either $a(I)$ and $b(I)$ have different signs for some proper subset $I\subset [d]$ or that $a$ and $b$ lie in different regions of linearity for $f_I$ or $g_I$ for some $I\subseteq [d]$. Indeed, if $D(a)\neq D(b)$, then either $S_a^+\neq S_b^+$, or $I_{a, >}^{+} \neq I_{b, >}^{+}$, or $I_{a, <}^{+} \neq I_{b, <}^{-}$ . Suppose the first case holds. That is, there is an $I\subseteq [d]$ such that $a(I)>0$ and $b(I)<0$ or vice versa. Or, if $S_a^+ = S_b^+$ but $I_{a,>}^+\neq I_{b,>}^+$, then there exists an $I\subseteq I_a^+=I_b^+$ such that $f_I(x)=x(I)$ in a small neighborhood of $a$ but $f_I(x)=x(I_b^+\setminus I)$ in a small neighborhood of $b$, that is, $a$ and $b$ are not in the same region of linearity of $f_I$. Similarly, if $I_{a,>}^-= I_{b,>}^-$ is violated then $a$ and $b$ lie in different regions of linearity for $g_I$. In either case, it follows that $D(a) \neq D(b)$ implies that $a, b$ lie in different maximal cones of $\Delta.$
\end{proof}

\begin{figure}[h]
    \centering
    \includegraphics[scale=0.5]{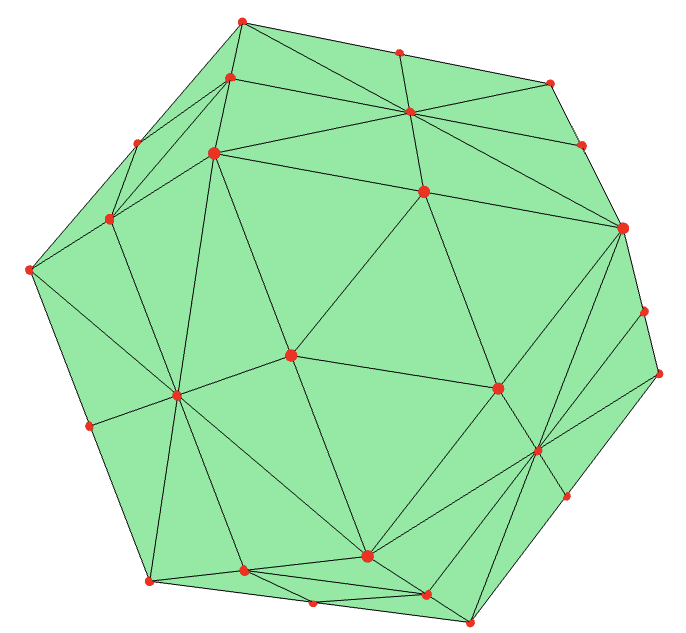}
    \caption{Bisection fan of the cuboctahedron, i.e. the root polytope of the lattice $A_{4}.$}
\label{fig:bisfan_rhombic_dodec}
\end{figure}

\section{Complexity of bisectors}\label{sec:complexity}
In this section, we investigate the number of maximal cells in the subdivision of the bisector which can be considered as a measure of its complexity. For a point $a\in \mathbb{R}^d$, we denote 
\[
N(a)=N_P(a) \ = \ |\{\text{maximal cells in $\bis \nolimits ^P (0,a)\}|$}.
\]
By Theorem~\ref{thm:generalposition}, for $a$ in general position, this is equivalent to
\[
N(a) \ = \ \{(F,G)\colon \, F,G \text{ distinct facets}, \, a\in \mathcal{B}_{F,G}\} \, .
\]
Let $P_{2n}$ denote a centrally symmetric polygon in $\mathbb{R}^{2}$ on $2n$ vertices, $C_{d} = [-1, 1]^{d}$ the hypercube, $\diamond_{\tiny d}=\conv \{\pm e_{i}: i \in [d]\}$ the cross-polytope and $Q_{d} = \conv\{e_{i}-e_{j}: 1 \leq i \neq j \leq d\}$ the root polytope of type $A$ studied in the previous sections.
We then have the following bounds on the number of maximal cells. Note that the number of maximal cells in case of the polygon and the hypercube is precise. Further, in the polygon case, the number is independent of the metric properties of the polygon.

\begin{thm}\label{prop:complexity}
    Let $a \in \mathbb{R}^{d}$ be a generic point. Then \begin{enumerate}
        \item[(1)] $N_{P_{2n}}(a) = 2n-1.$ 
        \item[(2)] $N_{C_{d}}(a) = d^{2}-d+1.$
        \item[(3)] $N_{\diamond_{d}}(a) = \Omega(2^{d}).$
        \item[(4)] $N_{Q_{d}}(a) =  \Omega(2^{d}).$
    \end{enumerate}
\end{thm}
\begin{proof}
We handle each case separately, using the same notation as in the corresponding previous sections.

$(1)$: By Lemma \ref{lem:2dim_biscone_decomp} together with Corollary \ref{bfg_swap}, for any $1\leq i \neq j \leq 2n$, 
\[
    \mathbb{R}^{2} = \bigcup_{k=0}^{2n-1} \mathcal{B}_{i+k, j+k}
    \]
    defines a decomposition of $\mathbb
{R}^{2}$ into bisection cones with disjoint interiors where addition in each index is taken modulo $2n$. We observe that each such decomposition of $\mathbb{R}^{2}$ is completely determined by the value of $(i-j) \mod 2n$, of which there are $2n-1$ choices. In each partition, by the disjoint interior condition on the decomposition, a point $a$ in general position lies in exactly one bisection cone. This gives us a total of $2n-1$ bisection cones containing $a$ and hence $N_{P_{2n}}(a)=2n-1.$ 

$(2)$: By Proposition \ref{cube:bis_cone}, for all $i,j\in [d]$, $i\neq j$, we have $a\in \mathcal{B}_{i^{\sigma _1},j^{\sigma _2}}$ if and only if $\sigma_{1} a_{i} \geq 0$ and $(-\sigma_{2})a_{j} \geq 0$. That is, given $i\neq j$ and $a$ in general position there is exactly one choice for $\sigma _1,\sigma_{2}\in \{+,-\}$ for which $a\in \mathcal{B}_{i^{\sigma _1},j^{\sigma _2}}$. This amounts to $d(d-1)$ bisection cones in which $a$ is contained in, one for every choice of pair $(i,j)$ such that $i\neq j$. Furthermore, each $a$ in general position is contained in exactly one facet cone $a\in \mathcal{B}_{i^{\sigma _1},i^{\sigma _2}}=C_{F_i^{\sigma _1}}$. Thus, in total, every generic $a\in \mathbb{R}^d$ is contained in exactly $d(d-1)+1$ bisection cones.

$(3)$: By Proposition~\ref{bis_tree}, for all subsets $I,J\subseteq [d]$ we have that $a \in \mathcal{B}_{I, J}$ if and only if the following four conditions are satisfied.

\begin{enumerate}
        \item[(i)] $a_{i} \geq 0 \quad \text{ for all } \, i \in I \cap J^{c} \, ,$
        \item[(ii)] $a_{i} \leq 0 \quad \text{ for all }  \, i \in J \cap I^{c} \, ,$
        \item[(iii)] $a(J) \leq a(J^{c})\, ,$
        \item[(iv)] $a(I^{c}) \leq a(I).$
    \end{enumerate}
    We claim that these conditions are satisfied for every $I \in I_{a, >}^{+}$ and $J \in I_{a, <}^{-}.$ Indeed, since $I \subseteq I_{a}^{+}$ and $J \subseteq I_{a}^{-}$, $a_i\geq 0$ for all $i\in I\subseteq I_a^+$ and $a_i\leq 0$ for all $i\in J\subseteq I_a^-$. Hence conditions $(i)$ and $(ii)$ hold. For condition $(iii)$, we note that since $J \in I_{a, <}^{-},$ we have 
    \[
    a(J) \leq a(I_{a}^{-}\setminus J)\leq a(I_{a}^{-}\setminus J) + a(I_{a}^{+}) = a(J^{c})\, .
    \]
 Similarly, since $I \in I_{a, >}^{+}$, condition $(iv)$ is satisfied:
 \[
 a(I) \geq a(I_{a}^{+} \setminus I) \geq a(I_{a}^{+} \setminus I) + a(I_{a}^{-}) = a(I^{c})\, .
 \]

It follows that $N_{\diamond_{\tiny d}}(a) \geq |I_{a, >}^{+}| \cdot |I_{a, <}^{-}|.$ If $|I_{a}^{+}| = l$, then $|I_{a, >}^{+}| = 2^{l-1}$ since for every subset $I$ of $I_a^+$, either $I$ or $I_{a}^{+}\setminus I$ is contained in $I_{a, >}^{+}$. Similarly, we have $|I_{a, <}^{-}| = 2^{d-l-1}$ which implies that $N_{\diamond_{\tiny d}}(a) \geq 2^{d-2}$ and thus $N_{\diamond_{\tiny d}}(a)=\Omega (2^d)$.

$(4)$: As before, we consider pairs $(I, J) \in I_{a, >}^{+} \times I_{a, <}^{-}.$ All such pairs satisfy $I \cap J = \emptyset$ and $I^{c} \cap J^{c} \neq \emptyset$ except when $I = I_{a}^{+}$ and $J = I_{a}^{-}.$ Any such pairs satisfy conditions of type II in Proposition~\ref{biscone_kd}, namely 
\begin{enumerate}
        \item[(i)] $a_{i} \geq 0 \quad \forall \, i \in I$
        \item[(ii)] $a_{i} \leq 0 \quad \forall \, i \in J$
        \item[(iii)] $a(J) \leq a(K) \quad \forall \, K \subseteq I^c \cap J^c$
        \item[(iv)] $a(K) \leq a(I) \quad \forall \, K \subseteq I^c \cap J^c$
\end{enumerate}

The first two conditions hold since $I\subseteq I_a^+$ and $J\subseteq I_a^-$. For the third condition, since $J \in I_{a, <}^{-}$, we have the following inequality for any $K \subseteq I^{c} \cap J^{c}$: 
\[
a(J) \leq a(I_{a}^{-} \setminus J) \leq a(K \cap (I_{a}^{-} \setminus J)) \leq a(K \cap (I_{a}^{-} \setminus J)) + a(K \cap (I_{a}^{+} \setminus J)) = a(K)  
\]
Condition (iv) holds by a similar chain of inequalities. Hence the number of pairs of subsets satisfying condition II is at least $|I_{a, >}^{+}|\cdot |I_{a, <}^{-}| - 1 =  2^{d-2}-1$ by the same arguments as for cross polytopes. For type III, we consider pairs $(I, J)$ such that $(I^{c}, J^{c})\in I_{a, <}^{-} \times I_{a, >}^{+}.$ This yields the same contribution as from type II, namely $2^{d-2}-1$. Hence, $N_{Q_d}\geq 2(2^{d-2} -1)$ and thus $N_{P}(a) = \Omega(2^{d})$. This completes the proof.
\end{proof}

\section{Outlook}\label{sec:outlook}
    We determined the bisection fans corresponding to the following polytopes: polygons, cubes, cross-polytopes, and the symmetric edge polytope of the complete graph. Taken together with the tropical unit ball, one might note that no single theme connects these five families of polytopes. Observe however, that this list contains two pairs of polar duals: namely cubes and cross-polytopes together with the symmetric edge polytope of the complete graph and the tropical unit ball. This might hint at the idea that if $P$ admits a bisection fan, then so does its polar dual $P^{\small \Delta}.$

    Having proved that the bisection fan exists for all centrally symmetric polytopes of dimension $2$, it is natural to ask the same question in dimension $3$ or higher. 
    \begin{quest}
        Does every centrally symmetric polytope admit a bisection fan?
    \end{quest}
    
    With the following pictures (Figure \ref{fig:panel_4_bisfans_fedorov}) as evidence, we posit that the bisection fan exists for the five Fedorov solids, that is the family of zonotopes that tile $3$-space by translations. This family consists of the cube (Figure \ref{fig:bisfan_cube}), the rhombic dodecahedron, which realises the tropical unit ball in dimension $3$ (Figure \ref{fig:rhombic_dodeca}), the hexagonal prism (Figure \ref{fig:hexa_prism}), the truncated octahedron, which realises the permutahedron of dimension $3$ (Figure \ref{fig:truncated_octa}), and the hexarhombic dodedecahedron (Figure \ref{fig:elongated_dodeca}). Note that in the cases where we have not proved the existence of the bisection fan, these pictures, generated by \texttt{polymake} \cite{gawrilow1997polymake}, represent the intersection of all the bisection cones of $P$ with the polytope $P$. The scripts that produced these pictures along with other supplementary material are freely available at 
    \begin{center}
    \url{https://github.com/AryamanJal/Bisection_fan}. 
    \end{center}
    From these pictures one could ask if the bisection cones - and fan - in these cases have combinatorial interpretations given the interesting subdivisions of the boundary that they give rise to. Observe from Figure~\ref{fig:panel_4_bisfans_fedorov} that even under simple operations, the bisection fan can change drastically. For example, the bisection fan of the hexagonal prism, if it exists, bears no obvious relation to that of the hexagon. This reinforces the leitmotif that simple convex bodies can give rise to bisectors with complex combinatorics.

    \begin{figure}
        \centering
        \begin{subfigure}[b]{0.475\textwidth}
            \centering
            \includegraphics[scale=0.5]{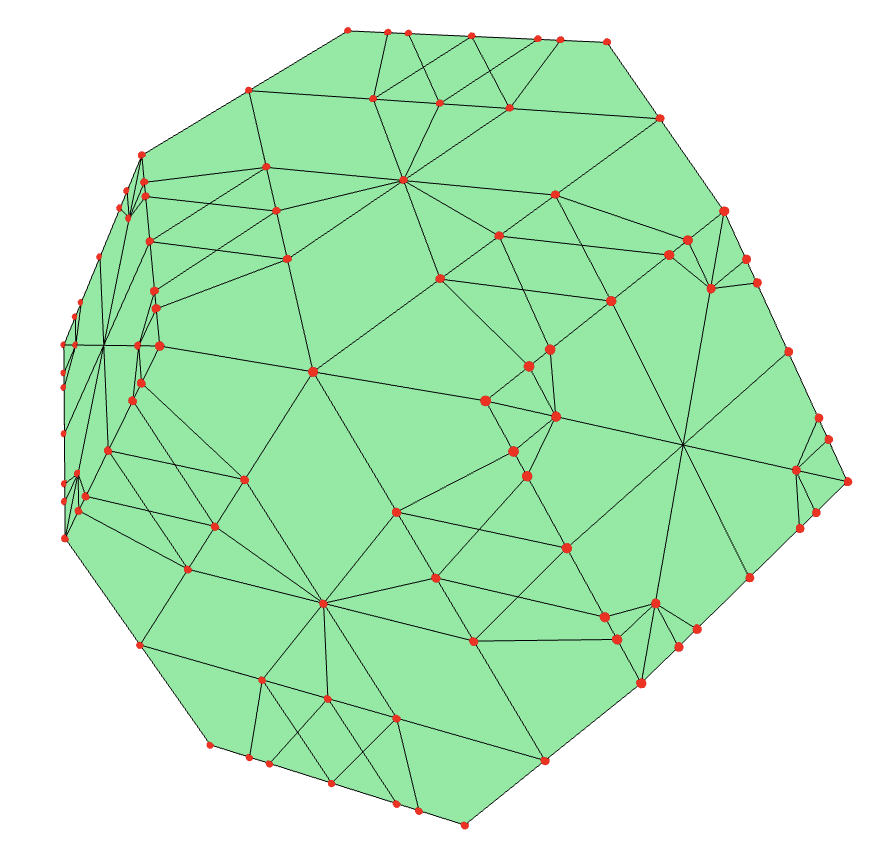}
    \caption{Truncated octahedron}
    \label{fig:truncated_octa}
        \end{subfigure}
        \hfill
        \begin{subfigure}[b]{0.475\textwidth}  
            \centering 
             \includegraphics[scale=0.5]{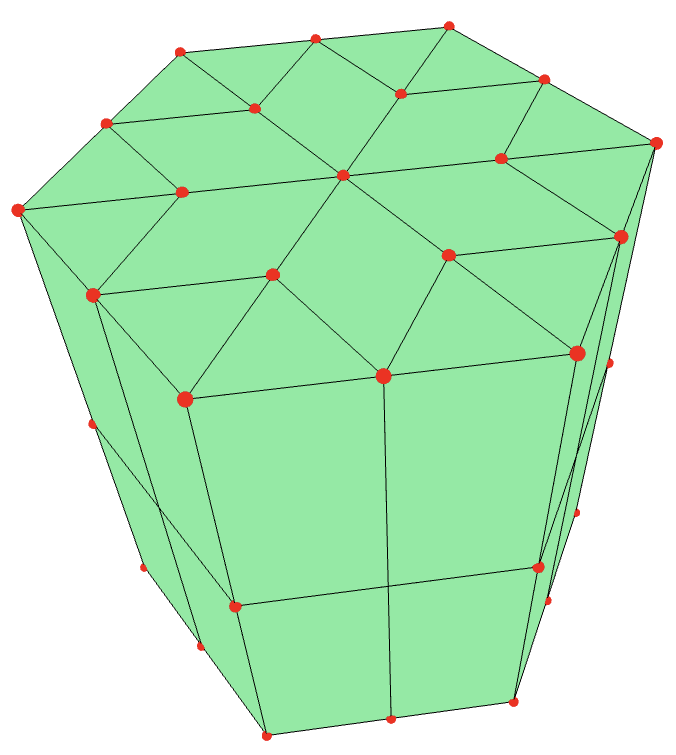}
    \caption{Hexagonal prism}
    \label{fig:hexa_prism}
        \end{subfigure}
        \vskip\baselineskip
        \begin{subfigure}[b]{0.475\textwidth}   
            \centering 
             \includegraphics[scale=0.5]{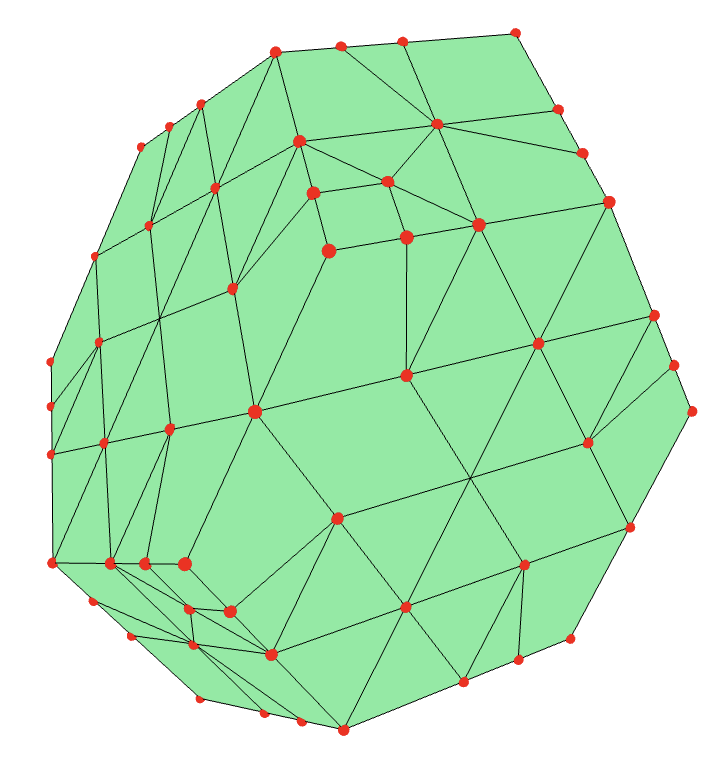}
    \caption{Hexarhombic dodecahedron}
    \label{fig:elongated_dodeca}
        \end{subfigure}
        \hfill
        \begin{subfigure}[b]{0.475\textwidth}   
            \centering 
            \includegraphics[scale=0.5]{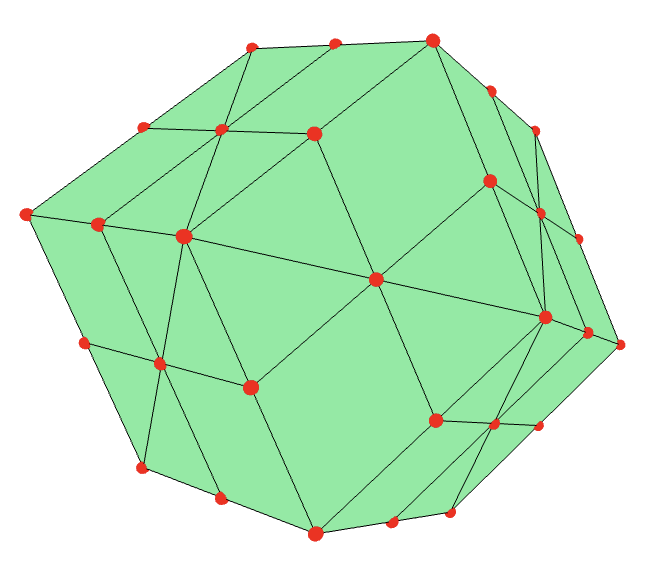}
    \caption{Rhombic dodecahedron}
    \label{fig:rhombic_dodeca}
        \end{subfigure}
        \caption[]
        {\small The intersection of $P$ and the collection of its bisection cones, for four of the five Fedorov solids $P$.} 
        \label{fig:panel_4_bisfans_fedorov}
    \end{figure}

    \textbf{Acknowledgements}: The authors are grateful to the anonymous reviewer for careful reading of the manuscript and helpful comments. The authors would like to thank Michael Joswig, Benjamin Schr\"oter and Francisco Criado for helpful and inspiring discussions. In particular, they are grateful to Michael Joswig for hosting them at TU Berlin during spring 2022 during which parts of this work were undertaken. They are also very thankful to Benjamin Schr\"oter for his invaluable help with $\texttt{polymake}$. 
    
    AJ and KJ were partially supported by the Wallenberg AI, Autonomous Systems and Software Program funded by the Knut and Alice Wallenberg Foundation. KJ was moreover supported by a Hanna Neumann Fellowship of the Berlin Mathematics Research Center Math+, grant 2018-03968 of the Swedish Research Council and the Göran Gustafsson Foundation.

\newpage
\bibliography{bisbib}
\bibliographystyle{plain}
\end{document}